\documentclass[a4paper, 12pt]{article}
\usepackage{amssymb,amsmath,amsthm}
\usepackage{fullpage}
\usepackage{hyperref}
\usepackage{lmodern}
\usepackage{eucal}
\usepackage{color}
\usepackage{stackrel}

\usepackage{mathrsfs,tikz,hyperref,
csquotes,enumerate,subfigure,bbm} 
\hypersetup{pdfborder={0000}, colorlinks=true, linkcolor=blue,citecolor=citegreen}
\definecolor{citegreen}{rgb}{0.2,0.2,0.6}
\usepackage[babel=true,kerning=true]{microtype}

\newcommand\sm{\setminus}
\newcommand\kp{\kappa}

\hypersetup{
 colorlinks=true,
 citecolor=darkred,
 linkcolor=blue,
 urlcolor=blue}

\renewcommand\Im{\mathrm{Im}\,}
\newcommand\bfE{\mathbf{E}}
\newcommand\bfH{\mathbf{H}}
\newcommand\M{\mathsf{M}}

\renewcommand\iff{{\textit{if and only if }}}

\usepackage[shortlabels]{enumitem}

\usepackage[font=small]{caption}

\usepackage{palatino}
\usepackage{eucal}
\usepackage{verbatim}

\usepackage{amsmath,amssymb,amsthm,graphicx,epstopdf,mathrsfs,url}

\hypersetup{
 colorlinks=true,
 citecolor=darkred,
 linkcolor=blue,
 urlcolor=blue}

\definecolor{darkred}{rgb}{0.4,0.1,0.1}
\definecolor{darkblue}{rgb}{0.1,0.1,0.4}
\definecolor{lightdarkblue}{rgb}{0.1,0.1,0.6}
\definecolor{darkgrey}{rgb}{0.5,0.5,0.5}
\usepackage{pgfplots}               

\newcommand{\comm}[1]{}

\newcommand\sff{{\mathsf{f}}} 
\newcommand\zz{{\mathsf{z}}} 

\definecolor{darkred}{rgb}{0.5,0.1,0.1}

\newcommand\scH{\mathscr{H}}
\newcommand\arr{\rightarrow}

\newcommand\lm{\lambda}

\newcommand\s{\sigma}

\newcommand\ii{{\mathsf{i}}}
\newcommand\p{\partial}

\newcommand\one{\mathbbm{1}}

\newcommand\dd{{\mathsf{d}}}

\newcommand\omg{\omega}

\def\softness{0.4}

\definecolor{softred}{rgb}{1,\softness,\softness}
\definecolor{softgreen}{rgb}{\softness,1,\softness}
\definecolor{softblue}{rgb}{\softness,\softness,1}
\definecolor{softrg}{rgb}{1,1,\softness}
\definecolor{softrb}{rgb}{1,\softness,1}
\definecolor{softgb}{rgb}{\softness,1,1}

\newcounter{counter_a}
\newenvironment{myenum}{\begin{list}{{\rm(\roman{counter_a})}}%
{\usecounter{counter_a}
\setlength{\itemsep}{1.ex}\setlength{\topsep}{1.5ex}
\setlength{\leftmargin}{5ex}\setlength{\labelwidth}{5ex}}}{\end{list}}

\usepackage[latin1]{inputenc}
\usepackage[T1]{fontenc}

\newcommand{\eg}{{\it e.g.}\,}
\newcommand{\ie}{{\it i.e.}\,}
\newcommand{\cf}{{\it cf.}\,}

\numberwithin{figure}{section}
\numberwithin{equation}{section}
\theoremstyle{plain}
\newtheorem*{thm*}{Theorem}
\newtheorem{thm}{Theorem}[section]

\newtheorem{lem}[thm]{Lemma}
\newtheorem{prop}[thm]{Proposition}

\newtheorem{dfn}[thm]{Definition}
\theoremstyle{remark}
\newtheorem{remark}[thm]{Remark}
\theoremstyle{plain}

\newcommand{\supp}{\mathrm{supp}\,}
\newcommand{\beu}{\begin{equation*}}
\newcommand{\eeu}{\end{equation*}}
\newcommand{\besu}{\begin{equation*}
\begin{aligned}}
\newcommand{\eesu}{\end{aligned}
\end{equation*}}
\newcommand{\bes}{\begin{equation}
\begin{aligned}}
\newcommand{\ees}{\end{aligned}
\end{equation}}

\newcommand\cB{\mathcal B}

\newcommand\cH{\mathcal H}

\newcommand\frh{\mathfrak h}

\newcommand\ov{\overline}
\newcommand\wt{\widetilde}
\newcommand\wh{\widehat}

\newcommand\void[1]{}

\newcommand\eps{\varepsilon}
\newcommand\ran{{\rm ran\,}}

      \def\dC{{\mathbb C}}

   \def\dN{{\mathbb N}}   
      \def\dR{{\mathbb R}}

   \def\dZ{{\mathbb Z}}

   \def\cB{{\mathcal B}}   
      
   \def\cH{{\mathcal H}}   
      
      \def\cO{{\mathcal O}}

\renewcommand{\div}{\mathrm{div}\,}
\newcommand{\grad}{\mathrm{grad}\,}

\newcommand{\dom}{\mathrm{dom}\,}

\renewcommand{\one}{\mathbbm{1}}
\numberwithin{equation}{section}
\theoremstyle{plain}

\renewcommand\H{\mathsf{H}}
\newcommand\T{\mathsf{T}}
\newcommand\E{\mathsf{E}}

\newcommand{\spann}{\mathrm{span}\,}
\newcommand{\Ups}{\Upsilon}
\newcommand{\Th}{\Theta}

\newcommand{\TE}{{\bf{TE}}}
\newcommand{\TM}{{\bf{TM}}}
\newcommand{\xx}{{\textsf{x}}}
\newcommand{\yy}{{\textsf{y}}}
\newcommand{\sfa}{\textsf{a}}
\newcommand{\sfb}{\textsf{b}}

\newcommand{\cl}{\mathrm{cl}}

\renewcommand{\tt}{\theta}
\renewcommand{\aa}{\alpha}
\newcommand\bfo{\mathbf{0}}

\newcommand{\curl}{\mathrm{curl}\,}
\newcommand{\divergence}{\mathrm{div}\,}

\usepackage[normalem]{ulem}
\definecolor{DarkGreen}{rgb}{0,0.5,0.1} 

\newcommand\soutD{\bgroup\markoverwith
	{\textcolor{DarkGreen}{\rule[.5ex]{2pt}{1pt}}}\ULon}

\newcommand{\Hm}[1]{\leavevmode{\marginpar{\tiny%
			$\hbox to 0mm{\hspace*{-0.5mm}$\leftarrow$\hss}%
			\vcenter{\vrule depth 0.1mm height 0.1mm width \the\marginparwidth}%
			\hbox to
			0mm{\hss$\rightarrow$\hspace*{-0.5mm}}$\\\relax\raggedright #1}}}

\renewcommand{\div}{\mathrm{div}\,}
\begin{document}

\title{\textbf{\Large
Spectral analysis of photonic crystals made of thin rods
	}}
\author{Markus Holzmann\footnote{
Institut f\"{u}r Numerische Mathematik,
Technische Universit\"{a}t Graz,
Steyrergasse 30, A 8010 Graz, Austria,
\texttt{holzmann@math.tugraz.at}} \and
Vladimir Lotoreichik\footnote{
Department of Theoretical Physics,
Nuclear Physics Institute, Czech Academy of Sciences, 250 68,
\v{R}e\v{z} near Prague, Czechia,
\texttt{lotoreichik@ujf.cas.cz}}}
\maketitle


%
\begin{abstract}
  In this paper we address the question how to design photonic crystals
  that have photonic band gaps around a finite number of given frequencies.
  In such materials electromagnetic waves with these frequencies can not propagate; 
  this makes them interesting for a large number of applications.
  We focus on crystals made of periodically ordered thin rods 
  with high contrast dielectric properties.
  We show that the material parameters can be chosen in such a way 
  that transverse magnetic modes  with given frequencies can not propagate in the crystal. 
  At the same time, for any  frequency belonging to a predefined range there exists a transverse electric mode 
  that can propagate in the medium.
  These results are related to the spectral properties of a weighted Laplacian and
  of an elliptic operator of divergence type both acting in $L^2(\dR^2)$.
  The proofs rely on perturbation theory of linear operators, Floquet-Bloch
  analysis, and properties of Schr\"odinger operators with point interactions.
\end{abstract}



\section{Introduction} 
\label{section_motivation}

\subsection{Motivation}
Photonic crystals gained a lot of attention in the recent decades both from the physical and the mathematical side.
An electromagnetic wave can propagate in the crystal
\iff its frequency does not belong to a photonic band gap. 
Therefore, photonic crystals can be seen as an optical analogue of semiconductors giving a physical motivation to study them.
The idea of designing periodic dielectric materials with photonic band gaps was proposed in~\cite{john_1987, yablonovitch_1987}. 
In the recent years 
a great advance in fabrication of such crystals was achieved.

Despite a substantial progress in the physical and mathematical investigation of photonic crystals 
with a given geometry
(see~\cite{DLPSW, joannopoulos_2008} and the references therein), 
the important task of designing photonic crystals
having band gaps of a certain predefined structure
still remains  challenging.
In this paper we are interested in the following 
inverse problem:
\[	
	\begin{matrix}
	\text{\emph{How to design a photonic crystal such that finitely many}}\\
	\text{\emph{predefined frequencies
	$\omega_1, \dots, \omega_N$ belong to photonic band gaps?}}\end{matrix}
\]

In order to tackle this problem we employ a special class of photonic crystals made of 
very thin infinite rods with high dielectric permittivity embedded into vacuum. To be more precise, 
let $\wh\Gamma$ be a parallelogram in $\dR^2$ which is spanned by two linearly independent vectors and let the points 
$\yy^{(1)}, \dots, \yy^{(N)} \in \wh{\Gamma}$ be pairwise distinct. 
The basis cell of the crystal consists of $N$ infinite rods with large relative dielectric permittivities 
$\eps^{(n)} \gg 1$, $n = 1, \dots, N$, whose cross sections 
are small bounded domains in $\dR^2$
localized near the points $\yy^{(n)}$, $n = 1, \dots, N$,
and surrounded by vacuum with relative dielectric permittivity $\varepsilon = 1$. 
Then the crystal is built by repeating the basis cell in a periodic way such that the whole Euclidean space~$\dR^3$ is filled; 
\cf Figure~\ref{figure_crystal} and 
Subsection~\ref{section_main_result} for a more precise definition.

Special crystals of the above type have already been investigated 
in~\cite{joannopoulos_2008, joannopoulos_villeneuve_97,
maradudin_93, meade_93, plihal_1991} by physical and numerical experiments.
They were one of the first photonic crystals treated in the literature because they are comparably simple to produce.
The results in the above mentioned papers 
indicate that these crystals may have band gaps for electromagnetic waves polarized in a special way. 
Our goal is to provide an analytic proof of this and related results.

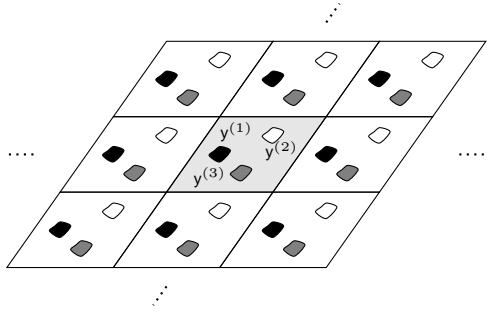
\begin{figure}[h!] 
	\centering
  \begin{tikzpicture}[xscale=0.7, yscale=0.5]
      \definecolor{fillcolor}{rgb}{0.900000,0.9,0.900000}
      \filldraw[fill=fillcolor,draw=black] (-1, -1)--(1, -1)--(2, 1)--(0, 1)--cycle;
      \filldraw[fill=black,draw=black] plot [smooth cycle] coordinates {(0,-0.2) (0.2,0.1) (0.01,0.21) (-0.1,0.1) (-0.2,-0.03)};
      \draw (-0.2, 0.6)
          node[right] {\tiny $\yy^{(1)}$};;
      \filldraw[fill=white,draw=black] plot [smooth cycle] coordinates {(1,0.3) (1.2,0.6) (1.01,0.71) (0.9,0.6) (0.8,0.47)};
      \draw (0.65, 0.1)
          node[right] {\tiny $\yy^{(2)}$};;
      \filldraw[fill=gray,draw=black] plot [smooth cycle] coordinates {(0.4,-0.7) (0.6,-0.4) (0.41,-0.29) (0.3,-0.4) (0.2,-0.53)};
      \draw (-0.7, -0.6)
          node[right] {\tiny $\yy^{(3)}$};;

      \draw (1, -1)--(3, -1)--(4, 1)--(2, 1)--cycle;
      \filldraw[fill=black,draw=black] plot [smooth cycle] coordinates {(2,-0.2) (2.2,0.1) (2.01,0.21) (1.9,0.1) (1.8,-0.03)};
      \filldraw[fill=white,draw=black] plot [smooth cycle] coordinates {(3,0.3) (3.2,0.6) (3.01,0.71) (2.9,0.6) (2.8,0.47)};
      \filldraw[fill=gray,draw=black] plot [smooth cycle] coordinates {(2.4,-0.7) (2.6,-0.4) (2.41,-.29) (2.3,-0.4) (2.2,-0.53)};

      \draw (-3, -1)--(-1, -1)--(0, 1)--(-2, 1)--cycle;
      \filldraw[fill=black,draw=black] plot [smooth cycle] coordinates {(-2,-0.2) (-1.8,0.1) (-1.99,0.21) (-2.1,0.1) (-2.2,-0.03)};
      \filldraw[fill=white,draw=black] plot [smooth cycle] coordinates {(-1,0.3) (-0.8,0.6) (-0.99,0.71) (-1.1,0.6) (-1.2,0.47)};
      \filldraw[fill=gray,draw=black] plot [smooth cycle] coordinates 
          {(-1.6,-0.7) (-1.4,-0.4) (-1.59,-0.29) (-1.7,-0.4) (-1.8,-0.53)};

      \draw (2, 1)--(4, 1)--(5, 3)--(3, 3)--cycle;
      \filldraw[fill=black,draw=black] plot [smooth cycle] coordinates {(3,1.8) (3.2,2.1) (3.01,2.21) (2.9,2.1) (2.8,1.97)};
      \filldraw[fill=white,draw=black] plot [smooth cycle] coordinates {(4,2.3) (4.2,2.6) (4.01,2.71) (3.9,2.6) (3.8,2.47)};
      \filldraw[fill=gray,draw=black] plot [smooth cycle] coordinates {(3.4,1.3) (3.6,1.6) (3.41,1.71) (3.3,1.6) (3.2,1.47)};

      \draw (0, -3)--(2, -3)--(3, -1)--(1, -1)--cycle;
      \filldraw[fill=black,draw=black] plot [smooth cycle] coordinates {(1,-2.2) (1.2,-1.9) (1.01,-1.79) (0.9,-1.9) (0.8,-2.03)};
      \filldraw[fill=white,draw=black] plot [smooth cycle] coordinates {(2,-1.7) (2.2,-1.4) (2.01,-1.29) (1.9,-1.4) (1.8,-1.53)};
      \filldraw[fill=gray,draw=black] plot [smooth cycle] coordinates {(1.4,-2.7) (1.6,-2.4) (1.41,-2.29) (1.3,-2.4) (1.2,-2.53)};

      \draw (-2, -3)--(0, -3)--(1, -1)--(-1, -1)--cycle;
      \filldraw[fill=black,draw=black] plot [smooth cycle] coordinates 
          {(-1,-2.2) (-0.8,-1.9) (-0.99,-1.79) (-1.1,-1.9) (-1.2,-2.03)};
      \filldraw[fill=white,draw=black] plot [smooth cycle] coordinates {(0,-1.7) (0.2,-1.4) (0.01,-1.29) (-0.1,-1.4) (-0.2,-1.53)};
      \filldraw[fill=gray,draw=black] plot [smooth cycle] coordinates 
          {(-0.6,-2.7) (-0.4,-2.4) (-0.59,-2.29) (-0.7,-2.4) (-0.8,-2.53)};

      \draw (-4, -3)--(-2, -3)--(-1, -1)--(-3, -1)--cycle;
      \filldraw[fill=black,draw=black] plot [smooth cycle] coordinates 
          {(-3,-2.2) (-2.8,-1.9) (-2.99,-1.79) (-3.1,-1.9) (-3.2,-2.03)};
      \filldraw[fill=white,draw=black] plot [smooth cycle] coordinates 
          {(-2,-1.7) (-1.8,-1.4) (-1.99,-1.29) (-2.1,-1.4) (-2.2,-1.53)};
      \filldraw[fill=gray,draw=black] plot [smooth cycle] coordinates 
          {(-2.6,-2.7) (-2.4,-2.4) (-2.59,-2.29) (-2.7,-2.4) (-2.8,-2.53)};

      \draw (0, 1)--(2, 1)--(3, 3)--(1, 3)--cycle;
      \filldraw[fill=black,draw=black] plot [smooth cycle] coordinates {(1,1.8) (1.2,2.1) (1.01,2.21) (0.9,2.1) (0.8,1.97)};
      \filldraw[fill=white,draw=black] plot [smooth cycle] coordinates {(2,2.3) (2.2,2.6) (2.01,2.71) (1.9,2.6) (1.8,2.47)};
      \filldraw[fill=gray,draw=black] plot [smooth cycle] coordinates {(1.4,1.3) (1.6,1.6) (1.41,1.71) (1.3,1.6) (1.2,1.47)};

      \draw (-2, 1)--(0, 1)--(1, 3)--(-1, 3)--cycle;
      \filldraw[fill=black,draw=black] plot [smooth cycle] coordinates {(-1,1.8) (-0.8,2.1) (-0.99,2.21) (-1.1,2.1) (-1.2,1.97)};
      \filldraw[fill=white,draw=black] plot [smooth cycle] coordinates {(0,2.3) (0.2,2.6) (0.01,2.71) (-0.1,2.6) (-0.2,2.47)};
      \filldraw[fill=gray,draw=black] plot [smooth cycle] coordinates 
          {(-0.6,1.3) (-0.4,1.6) (-0.59,1.71) (-0.7,1.6) (-0.8,1.47)};

      \draw[dotted, thick] (4.5, 0) -- (5, 0);
      \draw[dotted, thick] (-3.5, 0) -- (-4, 0);
      \draw[dotted, thick] (2, 3.5) -- (2.25, 4);
      \draw[dotted, thick] (-1, -3.5) -- (-1.25, -4);
  \end{tikzpicture}
  \caption{Cross section of the crystal. The basis period cell is colored in gray and contains three rods ($N=3$).}
	\label{figure_crystal}
\end{figure}

\subsection{Maxwell's equations and propagation of electromagnetic waves}
\label{ssec:Maxwell_Operator}
Under conventional physical assumptions for our 
model (absence of currents and electric charges, 
linear constitutive relations and no magnetic 
properties, i.e. the relative magnetic permeability satisfies $\mu \equiv 1$)
and a suitable choice of the units, 
Maxwell's equations for time harmonic fields
take the following simple form (see \cite{figotin_kuchment_1996_2}):
\begin{equation} \label{Maxwell_equations}
  \div (\eps \bfE) = 0, 
  \qquad 
  \curl \bfE =  \frac{\ii \omg}{c} \bfH,
  \qquad 
  \div \bfH = 0, 
  \qquad 
  \curl \bfH = -\frac{\ii \omg}{c} \eps \bfE.
\end{equation}
Here, the three-dimensional vector fields $\bfE$ and $\bfH$ are
the electric and the magnetic field, respectively, 
$\omega > 0$ is the frequency of the wave, $\varepsilon$
is the relative dielectric permittivity and $c > 0$ stands for the speed of light. 
Choose a system of coordinates such that the $x_3$-axis is parallel to the rods
building our photonic crystal.
In this paper, we are interested in $x_3$-independent waves, \ie
$\bfE = \bfE(x_1, x_2)$ and $\bfH = \bfH(x_1, x_2)$.
This assumption is reasonable, 
as the physical parameters also depend only on $x_1$ and $x_2$.
In the physical literature such waves are often called standing waves, 
as they propagate strictly parallel to the 
$x_1$-$x_2$-plane perpendicular to the rods and not in the $x_3$-direction.
The frequency $\omega > 0$ belongs to a photonic band gap
if the Maxwell equations~\eqref{Maxwell_equations} possess 
no bounded solutions.

Maxwell's equations can be regarded as a generalized eigenvalue problem for the Maxwell operator
\begin{equation*}
	  \M \begin{pmatrix} \bfE \\ \bfH \end{pmatrix} 
		= \begin{pmatrix} 
  		0            & \ii c \eps^{-1} \curl \\ 
        -\ii c\, \curl & 0 
      \end{pmatrix} 
      \begin{pmatrix} \bfE \\ \bfH \end{pmatrix},
\end{equation*}
which is defined on an appropriate subspace of 
$L^2(\dR^2, \dC^3, \eps \dd x) \times L^2(\dR^2, \dC^3, \dd x)$
that takes the constraints $\divergence (\eps\bfE) = 0$ 
and $\divergence \bfH = 0$ into account; 
\cf \cite{figotin_kuchment_1996_2} for more details. 
According to~\cite[Sec. 7.1]{figotin_kuchment_1996_2}
the operator $\M$ is self-adjoint in 
$L^2(\dR^2, \dC^3, \eps \dd x) \times L^2(\dR^2, \dC^3, \dd x)$.
Using periodicity and the results of~\cite{klein_koines_seifert_2002} it can be shown 
that $\omega$ belongs to a photonic band gap of the crystal  
\iff $\omega\notin\s(\M)$.
Therefore, the existence and location of photonic band gaps can be analyzed
by means of the spectral analysis of $\M$.

Since $\eps$ is periodic in $x_1$ and $x_2$ and
independent of $x_3$, the operator $\M$ can be decomposed as 
$\M = \M_1 \oplus \M_2$, where $\M_1$ acts on 
so-called transverse magnetic 
\TM-modes having the form
\begin{equation*}
  \bfE = (0, 0, E_3)^{\top}, \qquad \bfH = (H_1, H_2, 0)^{\top},
\end{equation*}
and $\M_2$ acts on transverse electric \TE-modes given by
\begin{equation*}
  \bfE = (E_1, E_2, 0)^{\top}, \qquad \bfH = (0, 0, H_3)^{\top};
\end{equation*}
see \cite[Sec. 7.1]{figotin_kuchment_1996_2}. 
Therefore,
it holds $\s(\M) = \s(\M_1)\cup\s(\M_2)$ 
and it suffices to perform the spectral analysis of $\M_1$ and $\M_2$ separately to characterize $\s(\M)$.
The spectra of $\M_j$, $j=1,2$, have their own
independent physical
meaning when taking polarization of electromagnetic waves
to \TE- and \TM-modes into account.
Moreover, $\s(\M_j)$, $j=1,2$, are in simple direct correspondence with the spectra of certain scalar differential operators
on $L^2(\dR^2)$; \cf Subsection~\ref{section_main_result}.

Several mathematical approaches to treat the spectral problems for the operators $\M_j$, $j=1,2$, have been developed.
Purely numerical methods are elaborated in \eg~\cite{maradudin_93, meade_93, plihal_1991}. 
A combination of numerical and analytical methods
is suggested in~\cite{hoang_plum_wieners2009}.
For a wide class of geometries a method
based on boundary integral equations 
is efficient~\cite{ammari_kang_lee_2009, AKL09}.
An analytic approach for high contrast media is proposed
by A.~Figotin and P.~Kuchment
for crystals of a different geometry from ours
which are composed of periodically ordered vacuum bubbles surrounded by an optically dense material with very 
large dielectric permittivity and of small width. In a series of 
papers~\cite{figotin_kuchment_1994, figotin_kuchment_1996_1, figotin_kuchment_1996_2, figotin_kuchment_1998_1, 
figotin_kuchment_1998_2}
these authors showed that such crystals have an arbitrarily large number of photonic band gaps. 
Their approach largely inspired the methods used in the present paper.

Finally, we point out that topics of recent interest
in this active research field include the analysis of guided perturbations
of photonic crystals~\cite{AS04, BHPW14, KO12},
of materials with non-linear constitutive relations~\cite{EKE12, ELT17},
and of photonic crystals made of metamaterials~\cite{CL13, E10}.

\subsection{Notations and statement of the main results} \label{section_main_result}
In order to formulate our main results, we fix some notations.
We set $\bfo := (0,0) \in \dR^2$. For $\xx \in\dR^2$ and $r > 0$ we define
$B_r(\xx) := \{\yy\in\dR^2\colon |\yy - \xx| < r\}$.
For $\aa,\beta \ge 0$ and $A, B \subset\dR^2$ we use the notation
\[
	\aa A + \beta B := \{\aa \xx + \beta \yy \in\dR^2\colon \xx \in A, \yy \in B\}.
\]	
For a measurable set $\Omega\subset\dR^2$
we denote its Lebesgue measure by $|\Omega|$ 
and its characteristic function by~$\one_\Omega$. 
As usual, $L^1(\Omega)$ stands for the space of integrable functions
over $\Omega$.  For $f\in L^1(\Omega)$ we introduce the notation
$\langle f\rangle_\Omega = \int_\Omega f(x) \dd x$. 
The $L^2$-space over $\Omega \subset \dR^2$ with the usual inner product
is denoted by $\big(L^2(\Omega),(\cdot,\cdot)_{L^2(\Omega)}\big)$
and the $L^2$-based Sobolev spaces by $H^k(\Omega)$, $k=1,2$,
respectively. For a self-adjoint operator $\T$ in a Hilbert space
we denote its spectrum by $\s(\T)$ and its resolvent set by $\rho(\T) := \dC \sm \s(\T)$.

Let $N \in \dN$ be fixed and let $\lm_1, \dots, \lm_N \in (0, \infty)$ be pairwise distinct;
without loss of generality we assume that $0 < \lambda_1 < \lambda_2 < \dots < \lambda_N$. 
These numbers are associated to the frequencies 
$\omg_1, \dots, \omg_N$ that are desired to be contained in photonic 
band gaps of the crystal via the relations $\lm_n = \left(\frac{\omega_n}{c}\right)^2$. 
Moreover, let $\sfa_1,\sfa_2 \in \dR^2$ be linearly independent. 
We set
\begin{equation} \label{Lambda_and_cell}
  \Lambda := 
  \big\{ n_1 \sfa_1 + n_2 \sfa_2\in\dR^2\colon (n_1, n_2) \in \dZ^2 \big\}
  \qquad\text{and}\qquad
  \wh{\Gamma} := \big\{ s_1 \sfa_1 + s_2 \sfa_2 \in\dR^2 \colon s_1, s_2 \in [0, 1) \big\}.
\end{equation}
For the points in $\Lambda$ we often use the notation $\yy_n = n_1 \sfa_1 + n_2 \sfa_2$, $n = (n_1, n_2) \in \dZ^2$.
%
%
%
%
%
%
We choose pairwise distinct points $\yy^{(1)}, \dots, \yy^{(N)} \in \wh\Gamma$ and define 
\begin{equation}\label{def:Y}
	Y := \big\{ \yy^{(1)}, \dots, \yy^{(N)} \big\}.
\end{equation}
Let $\Omega \subset \dR^2$ be a bounded domain
with $\bfo \in \Omega$ and let $r > 0$ be sufficiently small such that
\begin{equation*}
  \big(\yy^{(n)} + r \Omega\big) \cap \big(\yy^{(m)} + r \Omega\big) 
  = \varnothing, \qquad n \neq m.
\end{equation*}
For $n = 1, \dots, N$ we define 
\begin{equation} \label{def_mu}
	\mu_n(x) 
	:= 
	\frac{2 \pi}{\lm_n |\Omega|} x + 
	c_n x^2,
\end{equation}
where $c_n \in \dR$, $n=1,\dots,N$, are some constant parameters.
Finally, we introduce the function $w_r\colon\dR^2\arr\dR$ by
\begin{equation} \label{def_eps}
  w_r
  := 
  1 + 
  \frac{1}{r^2}
  \sum_{n=1}^N\mu_n
  \left(\frac{1}{|\ln r|}\right) 
  \one_{\Lambda + \yy^{(n)} + r \Omega}.
\end{equation}
The relative dielectric permittivity $\eps_r\colon\dR^3\arr\dR$,
which describes the physical properties of the crystal,
is expressed through $w_r$ by
$\eps_r(x_1,x_2,x_3) := w_r(x_1,x_2)$.
In order to treat the spectral problem for the associated Maxwell operator 
$\M = \M_1\oplus \M_2$ described in Subsection~\ref{ssec:Maxwell_Operator},
we introduce two partial differential operators in $L^2(\dR^2)$ by 
\begin{subequations}
\begin{align}
	   \Th_r f & := -w_r^{-1} \Delta f, 
	   &
	   \dom \Th_r  & := H^2(\dR^2),	
	  \label{def_Theta}\\
 	   \Ups_r f  & := -\div (w_r^{-1} \grad f), 
 	   &
       \dom \Ups_r & := 
        \big\{ f \in H^1(\dR^2) \colon \div (w_r^{-1} \grad f) 
        \in L^2(\dR^2) \big\}.\label{def_Upsilon}
\end{align}
\end{subequations}
%
%
According to \cite[Sec. 7.1]{figotin_kuchment_1996_2} we have 
\begin{equation*}
\begin{split}
	& \omg \in \s(\M_1)
	\quad \Longleftrightarrow \quad 
	\left( \frac{\omg}{c} \right)^2 \in \s(\Th_r),	\\
 	& \omg \in \s(\M_2)
	\quad \Longleftrightarrow \quad
	\left( \frac{\omg}{c} \right)^2 \in \s(\Ups_r).
\end{split}	
\end{equation*}

Following the strategy of~\cite{figotin_kuchment_1996_2,figotin_kuchment_1998_2}
in order to investigate the spectral properties of $\Th_r$ we introduce
a family of auxiliary Schr\"odinger operators 
\begin{equation}\label{def_H}
  \H_{r,\lm} f := -\Delta f - \lm (w_r - 1) f, 
  \qquad \dom \H_{r,\lm} := H^2(\dR^2),\qquad \lm\ge 0.
\end{equation}
It is not difficult to check that 
\[
	\lm \in \s(\Th_r)\quad \Longleftrightarrow \quad \lm\in\s(\H_{r,\lm}).
\]
We show that the Schr\"odinger operators $\H_{r,\lm}$ 
converge (as $r \arr 0+$) in the norm resolvent sense to Hamiltonians with point interactions supported on $Y + \Lambda$. 
This convergence result is already demonstrated in a more general setting 
in~\cite{behrndt_holzmann_lotoreichik_2014}, but for our special form of 
$w_r$ we provide a refined analysis of the approximation 
including an estimate for the order of convergence.
For similar results in the case of a single point interaction in $\dR^2$ 
and in other space dimensions see \cite{AGHH87, AGHK84, holden_hoegh_krohn_johannesen1984} 
or the monograph~\cite{albeverio_gesztesy_hoegh-krohn_holden} and the references therein. 
Using the known spectral properties of these limit operators with point interactions and continuity arguments 
one can prove that the initially given number $\lm_n$ belongs to a gap of 
$\s\big(\H_{r,\lm_n}\big)$, 
$n =  1, \dots, N$, if the geometry of the crystal and 
the parameters in the definition of $w_r$ are chosen appropriately. 
This leads to the existence of gaps in $\s(\Th_r)$ in the vicinities of 
$\lm_n$ which is the first main result of this paper
and whose proof is provided in Section~\ref{section_weighted}.
%
\begin{thm} \label{theorem_spectrum_Theta}
	There exist linearly independent vectors 
	$\sfa_1, \sfa_2 \in \dR^2$ and coefficients $c_1, \dots, c_N$
	such that
	\begin{equation*}
    	\bigcup_{n=1}^N 
    	\big(\lm_n - \lm_1 r^2|\ln r|, \lm_n + \lm_1 r^2|\ln r|\big) 
    	\subset \rho(\Th_r)
	\end{equation*}
	for all sufficiently small $r > 0$.
\end{thm}

Concerning the analysis of $\Ups_r$, there are several works 
on similar divergence type operators with high contrast coefficients,
see \eg \cite{hempel_lienau_2000, zhikov_2005}.
They have in common that the parameter becomes large on a domain whose diameter divided by the size of the period cell
is constant and thus, these results do not apply in our setting. 
Other closely related results on the spectral analysis of divergence
type operators with high contrast coefficients 
can be found \eg in~\cite{cherednichenko_cooper_2016, figotin_kuchment_1996_1, khrabustovskyi_2013, zhikov_2000}.

In our setting we use the Floquet-Bloch decomposition to show that any compact subinterval of $[0,\infty)$
is contained in the spectrum of~$\Ups_r$ for sufficiently small $r > 0$. 
This is the second main result of the paper; its proof is provided in Section~\ref{section_div}.
%
\begin{thm} \label{theorem_divergence}
  For any $L > 0$ there exists $r_0 = r_0(L) > 0$ 
  such that $[0, L] \subset \s(\Ups_r)$ for all $0 < r < r_0$.
\end{thm}
We conclude this section by a discussion
and interpretation of our results.  
According to Theorem~\ref{theorem_spectrum_Theta}, 
for a given set of pairwise distinct frequencies $\omg_1, \dots, \omg_N \in (0,+\infty)$ 
there exists a geometry of the crystal
(a suitable period cell and coefficients $c_1, \dots, c_N$) such that 
\begin{equation*}
  \left\{\left(\tfrac{\omg_1}{c} \right)^2, 
  \left(\tfrac{\omg_2}{c} \right)^2,\dots, \left(\tfrac{\omg_N}{c} \right)^2 
  \right\} \subset \rho(\Th_r), 
\end{equation*}
if the diameter of the rods (related to $r > 0$) is sufficiently small.
Moreover, we have an estimate for the size of the gap around
each  $\omg_n$.
In particular,
our results demonstrate a way how to construct photonic crystals such that
\TM-modes 
with frequencies in the vicinities of $\omg_n$, $n= 1,\dots, N$, can not propagate through it. 
At the same time, in view of Theorem~\ref{theorem_divergence}
there are no gaps in compact subintervals of the spectrum of $\Ups_r$ for small $r > 0$.
Restricting the frequencies to certain ranges, as it is typically the case in applications, 
there exists an $r_0 > 0$ such that for any frequency in this range
there is a \TE-mode with this frequency which can propagate through the crystal
for any $r\in (0,r_0)$.
These results perfectly match the experimental data and numerical 
tests in~\cite{joannopoulos_2008, joannopoulos_villeneuve_97, meade_93}
performed for the special case of the square lattice and $N=1$.

\subsubsection*{Organization of the paper}
In Section~\ref{section_delta} we introduce Schr\"odinger operators with point 
interactions supported on a lattice
and collect some results about their spectra. These results are employed 
in the spectral analysis of $\Th_r$ in Section~\ref{section_weighted}. 
Next, the operator $\Ups_r$ is investigated in Section~\ref{section_div}. 
Finally, Appendix~\ref{appendix_convergence_delta}
contains the technical analysis of the convergence of $\H_{r,\lm}$ in the norm resolvent sense to a Schr\"odinger operator with point 
interactions supported on a lattice. 

\section{Schr\"odinger operators with point interactions supported on lattices} \label{section_delta}

In this preliminary section we fix some notations that are associated to
lattices of points. Furthermore, we introduce Schr\"odinger operators with point 
interactions supported on a lattice and discuss
their spectral properties. These preparations will be useful in the
spectral analysis of $\Th_r$.

Let two linearly independent vectors $\sfa_1, \sfa_2 \in \dR^2$ be given and
let the lattice $\Lambda$ and the period cell $\wh \Gamma$ be defined
by~\eqref{Lambda_and_cell}.
Next, we introduce the associated \emph{dual lattice} $\Gamma$ by 
\begin{equation} \label{def_Gamma}
	\Gamma := \big\{ n_1 \sfb_1 + n_2 \sfb_2 \in\dR^2 \colon n_1, n_2 \in \dZ \big\},
\end{equation}
where $\sfb_1,\sfb_2 \in \dR^2$ are defined via
$\sfa_m \sfb_l = 2 \pi \delta_{ml}$ for  $m, l= 1, 2$.
The \emph{Brillouin zone} $\wh\Lambda \subset \dR^2$ corresponding to the lattice $\Lambda$ is defined by
\begin{equation} \label{Brillouin}
	\wh\Lambda := 
	\left\{ s_1 \sfb_1 + s_2 \sfb_2 \in\dR^2 \colon s_1, s_2 \in 
	\left[ -\tfrac{1}{2}, \tfrac{1}{2} \right] \right\}.
\end{equation}
%
In what follows we are going to discuss Hamiltonians with point interactions supported on $\Lambda$ following the lines 
of~\cite[Sec. III.4]{albeverio_gesztesy_hoegh-krohn_holden}.
Let $-\Delta$ be the self-adjoint free Laplacian in~$L^2(\dR^2)$ with the domain
$\dom(-\Delta)  = H^2(\dR^2)$. Its resolvent is denoted by
$R_0(\nu) := (-\Delta - \nu)^{-1}$.
For $\nu \in \rho(-\Delta) = \dC \sm [0, \infty)$
the integral kernel $G_\nu$ of $R_0(\nu)$ is given by
\begin{equation} \label{def_G_lambda}
	G_\nu(\xx - \yy) 
	= 
	\frac{\ii}{4} H^{(1)}_0\big(\sqrt{\nu} |\xx-\yy|\big),
\end{equation}
where $\text{Im} \sqrt{\nu} > 0$ and $H^{(1)}_0$ is the Hankel function of 
the first kind and order zero; \cf \cite[Chap.~9]{abramowitz_stegun} for details on Hankel functions.
Next, we set
\begin{equation} \label{def_G_tilde}
  \wt{G}_{\nu}(\xx) = 
  \begin{cases} G_\nu(\xx), &\quad\xx \neq 0, \\ 0, &\quad\xx = 0. \end{cases}
\end{equation}
For $\aa \in \dR$ and $m,l\in\dZ^2$ we define
\[
	q_{\aa, \Lambda}^{ml}(\nu) 
	:= 
	\left( \aa - \frac{1}{2 \pi} 
	\left(\gamma - \ln \frac{\sqrt{\nu}}{2 \ii}\right) \right)
	\delta_{ml} - \wt{G}_{\nu}(\yy_m - \yy_l),
\]
where $\gamma = 0.5772\dots$ is the Euler-Mascheroni constant
and $\yy_p = p_1 \sfa_1 + p_2 \sfa_2$ for $p = (p_1, p_2) \in \dZ^2$.
Eventually, we introduce for $\nu\in\dC\sm\dR$ the matrix 
\begin{equation} \label{def_Gamma_alpha}
  Q_{\aa,\Lambda}(\nu) 
  := 
  \left\{q_{\aa,\Lambda}^{ml}(\nu)   \right\}_{m, l \in \dZ^2},
\end{equation}
which induces a closed operator in $\ell^2(\dZ^2)$ that admits a bounded and everywhere defined inverse,
if $\Im \sqrt{\nu}$ is sufficiently large; 
\cf \cite[Thm. III.4.1]{albeverio_gesztesy_hoegh-krohn_holden}. 
We denote this operator again by $Q_{\aa,\Lambda}(\nu)$
as no confusion will arise. The matrix elements of
the inverse $Q_{\aa,\Lambda}(\nu)^{-1}$ 
in $\ell^2(\dZ^2)$ are denoted
by $r^{ml}_{\aa,\Lambda}(\nu)$.
\begin{dfn}\label{def:Op_PI}
	 {\rm The Schr\"odinger operator $-\Delta_{\aa, \Lambda}$ 
	 with point interactions}
	 supported on $\Lambda$
	  with coupling constant $\aa \in \dR$ is defined
	 as the self-adjoint operator 
	 in $L^2(\dR^2)$ with the resolvent 
	 \begin{equation} \label{def_delta_op_const}
		R_{\aa,\Lambda}(\nu) := (-\Delta_{\aa,\Lambda}  - \nu)^{-1} 
	  	= 
		R_0(\nu) 
	  	+
	  \sum_{m,l\in\dZ^2} 
	  r^{ml}_{\aa,\Lambda}(\nu)\,
	  \big(\,\cdot\, ,\, \ov{G_{\nu}(\cdot - \yy_l)}\,\big)_{L^2(\dR^2)} G_\nu(\cdot - \yy_m),
	\end{equation}
	where $\nu \in \dC \sm \dR$ and
	$\yy_p = p_1 \sfa_1 + p_2 \sfa_2$ for $p = (p_1,p_2)\in\dZ^2$.
\end{dfn}
Next, we are going to investigate the spectrum of $-\Delta_{\aa, \Lambda}$.
For this purpose, we introduce for $\alpha \in \mathbb{R}$ the numbers $E_j = E_j(\alpha, \Lambda)$,
$j \in \{ 0, 1, 2 \}$, as follows: $E_0$ is the smallest zero of the function\footnote{Eq.~\eqref{equation_E_0}
differs from the condition in \cite[Eq.~(4.42) in Sec.~III.4]{albeverio_gesztesy_hoegh-krohn_holden}, as the term $\frac{1}{2 \pi} \ln 2$
was forgotten there
(it disappeared in the convergence analysis in \cite[Eq.~(4.29) in Sec.~III.4]{albeverio_gesztesy_hoegh-krohn_holden}).}
\begin{equation} \label{equation_E_0}
  E\mapsto g(E, \bfo) + \frac{1}{2 \pi}(\gamma + \ln 2) - \alpha,
\end{equation}
where $g(E, \theta)$ is defined for 
$\tt \in \wh\Lambda$ and $E \notin \{ |\xx + \tt|^2\colon \xx \in \Gamma \}$
by%
\begin{equation*}
	g(E, \tt) 
	:= 
	\frac{1}{4 \pi^2} \lim_{R \arr \infty} 
	\Bigg[ \sum_{\xx \in \Gamma\colon |x + \tt| \leq R} 
	\frac{|\wh\Lambda|}{|x + \tt|^2 - E} - 2 \pi \ln R \Bigg].
\end{equation*}
Similarly, the number $E_1$ is given by the smallest 
zero of the function
\begin{equation} \label{equation_E_1}
  E\mapsto g(E, \theta_0) + \frac{1}{2 \pi}(\gamma + \ln 2) - \alpha,
\end{equation}
where $\theta_0:=-\frac{1}{2} (\sfb_1 + \sfb_2)$. Eventually, 
let $\sfb_- \in \{\sfb_1, \sfb_2\}$ be a vector satisfying $|\sfb_-| = \min\{ |\sfb_1|, |\sfb_2|\}$.
Then, we set $E_2 := \min\big\{ \wt{E}, \frac{1}{4}|\sfb_-|^2 \big\}$, where $\wt{E}$
is the smallest positive solution of equation \eqref{equation_E_0}\footnote{Note 
that $\wt{E}$ is equal to $E^{\alpha, \Lambda}_{\sfb_-}(\bfo)$ in the notation 
of \cite[Sec.~III.4]{albeverio_gesztesy_hoegh-krohn_holden}. 
The fact that $E^{\alpha, \Lambda}_{\sfb_-}(\bfo)$
is the smallest positive solution of equation \eqref{equation_E_0} can be shown in the same way as in the proof of
\cite[Thm.~III.1.4.4]{albeverio_gesztesy_hoegh-krohn_holden}. Observe that \eqref{equation_E_1} is modified similarly 
as \eqref{equation_E_0} compared to \cite{albeverio_gesztesy_hoegh-krohn_holden}.}.
All the numbers $E_0, E_1$, and $E_2$ are well defined; \cf 
\cite[Sec. III.4]{albeverio_gesztesy_hoegh-krohn_holden}.
In the next proposition we summarize some fundamental spectral properties of $-\Delta_{\aa, \Lambda}$ 
that can be found \eg in~\cite[Thm. III.4.7]{albeverio_gesztesy_hoegh-krohn_holden}.
%
\begin{prop} \label{thm_spectrum_delta_op}
  Let $\aa \in \dR$ and $\Lambda$ 
  be as in~\eqref{Lambda_and_cell}.
  Let the Schr\"odinger operator $-\Delta_{\aa,\Lambda}$ 
  be as in Definition~\ref{def:Op_PI} and
  let $\tt_0$, $g(\cdot, \cdot)$
  and $E_j = E_j(\aa,\Lambda)$, $j=0,1,2$, be as above.
  Then the following claims hold.
  %
  \begin{myenum}
  \item $\s(-\Delta_{\aa, \Lambda}) 
	    = [E_0, E_1] \cup [E_2, \infty)$.
   \item $E_0 < 0$ and $E_2 > 0$ for all $\aa \in \dR$.
 	\item $E_1 < 0$ 
    \iff $\aa < g(\bfo, \tt_0) + \frac{1}{2 \pi}(\gamma + \ln 2)$.\footnote{This condition
    differs from eq. (4.51) in \cite[Thm. III.4.7]{albeverio_gesztesy_hoegh-krohn_holden}, 
    the term $\frac{1}{2 \pi}(\gamma + \ln 2)$
    was forgotten; but it must be there; \cf~\cite[Eq.~(4.29) and (4.42) in Sec. III.4]{albeverio_gesztesy_hoegh-krohn_holden}.}
 	\item There exists an $\aa_1 = \aa_1(\Lambda) \in \dR$ 
	  such that $E_2 \le E_1$  for any  $\aa \ge \aa_1$. 
	  In particular,   $\s(-\Delta_{\aa, \Lambda}) 
 		 =[ E_0, +\infty )$ holds for all $\aa \ge \aa_1$.
  \end{myenum}
\end{prop}

By Proposition~\ref{thm_spectrum_delta_op} the operator $-\Delta_{\aa, \Lambda}$ has a gap in its spectrum, if the 
interaction strength is chosen in a proper way. In the rest of this section, we are going to investigate this gap in 
more detail. In particular, we will show that for a given compact interval $[a, b] \subset \dR$ there 
exist a lattice $\Lambda$ and an interaction strength $\aa$ such that $[a, b]$ is contained in the spectral gap of 
$-\Delta_{\aa, \Lambda}$. 
To this aim we introduce for $k > 0$ the unitary scaling operator
\begin{equation*} 
	U_k\colon L^2(\dR^2) \arr L^2(\dR^2), 
	\qquad 
	(U_k f)(\xx) := k^{-1} f \big( k^{-1} \xx \big).
\end{equation*}
Its inverse $U_k^{-1}\colon L^2(\dR^2) \arr L^2(\dR^2)$ clearly acts as 
$(U_k^{-1} f)(\xx) = k f \left( k \xx \right)$.
In the next proposition we show that this rescaling yields, up to multiplication with a constant, a unitary equivalence 
between point interaction operators with suitably modified
geometries of lattices
and  strengths of interactions.
\begin{prop} \label{prop_unitary_equiv}
	Let $\aa \in \dR$ and $\Lambda$ 
	be as in~\eqref{Lambda_and_cell}.
	For $k > 0$ set $\Lambda_k := k^{-1}\Lambda$
	and $\aa_k := \aa - \frac{\ln k}{2 \pi}$. 
  Let the Schr\"odinger operators $-\Delta_{\aa,\Lambda}$ 
  and $-\Delta_{\aa_k,\Lambda_k}$ 
  be as in Definition~\ref{def:Op_PI}.
	Then it holds
	\begin{equation*}
		U_k^{-1} \left(-\Delta_{\aa, \Lambda}\right) U_k 
		= k^{-2}\left(-\Delta_{\aa_k, \Lambda_k}\right).
	\end{equation*}
\end{prop}
\begin{proof}
	Let $\nu \in \dC \sm \dR$. 
	We show 
	\begin{equation*}
    	U_k^{-1}R_{\aa, \Lambda}(\nu) U_k 
    	= 
    	k^2 R_{\aa_k, \Lambda_k}(k^2\nu),
	\end{equation*}
	which yields then the claim. 
	By~\eqref{def_delta_op_const} it holds
	\begin{equation*}
    \begin{split}
    	U_k^{-1}R_{\aa,\Lambda}(\nu) U_k  
    	= 
		U_k^{-1}R_0(\nu) U_k  
	   	+ 
	   	\sum_{m,l \in \dZ^2} r^{ml}_{\aa,\Lambda}(\nu)\,
	   \big(\,\cdot\,,\, U_k^{-1} \ov{G_{\nu}(\cdot -  \yy_l)}\, \big)_{L^2(\dR^2)} 
            U_k^{-1} G_\nu(\cdot -  \yy_m).
    \end{split}
	\end{equation*}
	Since $U_k^{-1} (-\Delta - \nu) U_k = k^{-2}(-\Delta - k^2 \nu)$, we get
	\begin{equation}\label{eq:unitary1}
    	U_k^{-1} R_0(\nu) U_k = k^2 R_0(k^2\nu).
	\end{equation}
	Using the definition of $U_k^{-1}$ we obtain  
	for any $\yy \in \Lambda$ the relation
	\begin{equation}\label{eq:unitary2}
	U_k^{-1} G_{\nu}(\cdot -  \yy) = k G_{k^2 \nu} (\cdot - k^{-1} \yy )
        \end{equation}	 
	almost everywhere in $\dR^2$. This implies 
	\begin{equation*}
	\begin{split}
  		\big(\,\cdot\,,\, U_k^{-1} \ov{G_{\nu}(\cdot -  \yy)}\, \big)_{L^2(\dR^2)}
		&=
		\big(\,\cdot\,,\, k\ov{G_{k^2 \nu}(\cdot -  k^{-1} \yy)}\,\big)_{L^2(\dR^2)}.
    \end{split}
  \end{equation*}
  Eventually, a straightforward calculation yields
  \begin{equation}\label{eq:unitary3}
  \begin{split}
      q_{\aa, \Lambda}^{ml}(\nu) &
      = 
      \left(\aa - \frac{1}{2 \pi} 
      \left(\gamma - \ln \frac{\sqrt{\nu}}{2\ii}\right) \right)
      \delta_{ml} - \wt{G}_{\nu}(\yy_m - \yy_l) \\
      &=  
      \left( \aa_k - \frac{1}{2 \pi} 
      \left(\gamma - \ln \frac{k \sqrt{\nu}}{2\ii}\right)
      \right) \delta_{ml} 
      - 
      \wt{G}_{k^2 \nu}(k^{-1}( \yy_m - \yy_l))  
      = q_{\aa_k, \Lambda_k}^{ml}(k^2 \nu).
    \end{split}
  \end{equation}
  %
  Hence, the identity
  $r^{ml}_{\aa,\Lambda}(\nu) = r^{ml}_{\aa_k,\Lambda_k}(k^2\nu)$
  follows. Finally, employing~\eqref{eq:unitary1},~\eqref{eq:unitary2} and~\eqref{eq:unitary3}
  we get
  \[
    \begin{split}
      U_k^{-1}R_{\aa, \Lambda}(\nu) U_k 
      &
      = 
      U_k^{-1}R_0(\nu) U_k + \sum_{m, l \in \dZ^2} 
      r^{ml}_{\aa,\Lambda}(\nu)\,
      \big(\,\cdot\,,\,U_k^{-1} \ov{G_{\nu}(\cdot -  \yy_l)}\, \big)_{L^2(\dR^2)} U_k^{-1}
      G_{\nu}(\cdot -  \yy_m) \\
      &= k^2 R_0(k^2 \nu)
      + k^2\sum_{m, l \in \dZ^2} 
      r^{ml}_{\aa_k,\Lambda_k}(k^2\nu)\,
      \big(\,\cdot\,,\, \ov{G_{k^2 \nu}(\cdot -  k^{-1}\yy_l)}\, \big)_{L^2(\dR^2)}
      G_{k^2 \nu}(\cdot -  k^{-1}\yy_m)\\
      &= k^2R_{\aa_k, \Lambda_k}(k^2\nu).    \qedhere
    \end{split}
  \]
\end{proof}

The following useful statement
follows immediately from Propositions~\ref{thm_spectrum_delta_op} 
and~\ref{prop_unitary_equiv}.
\begin{prop} \label{thm_existance_spectral_gap_delta_op}
	Let $a,b\in\dR$ with $a < b$ be given.
	Then there exists a lattice 
	$\Lambda$ and a coupling $\aa \in \dR$ such that 
	the interval $[a, b]$ belongs to a gap of the spectrum 
	of the Schr\"odinger operator 
	$-\Delta_{\aa,\Lambda}$ in Definition~\ref{def:Op_PI}, i.e.
	\[	 
		[a, b] \subset \rho(-\Delta_{\aa, \Lambda}).
	\]
\end{prop}
\begin{proof}
	According to Proposition~\ref{thm_spectrum_delta_op}
	one can find a lattice 
	$\Lambda_0  = \big\{ n_1 \sfa_1 + n_2 \sfa_2\in\dR^2 \colon n_1, n_2 \in \dZ \big\}$ 
	and a coupling constant $\aa_0 \in \dR$ such that 
	\[
		0 \notin 
		\s(-\Delta_{\aa_0,\Lambda_0}) = [E_0, E_1] 
		\cup [E_2, \infty).
	\]
	Furthermore, by
	Proposition~\ref{prop_unitary_equiv} it holds for any $k >0$
	\begin{equation*}
		\s(-\Delta_{\aa_k,\Lambda_k}) =
		k^2\s(-\Delta_{\aa_0,\Lambda_0}) = 
		[k^2 E_0, k^2 E_1] \cup [k^2 E_2, \infty),
	\end{equation*}
	where
	$\aa_k = \aa_0 - \frac{1}{2 \pi} \ln k$ 
	and $\Lambda_k := k^{-1}\Lambda_0$.
	It remains to 
	choose the parameter $k > 0$  so large that 
	$k^2 E_1 < a < b < k^2 E_2$. Then the lattice
	$\Lambda = \Lambda_k$ and the coupling coefficient $\aa = \aa_k$
	fulfill all the requirements. 
\end{proof}
Finally, we define 
Schr\"odinger operators with point interactions
supported on a shifted lattice. For this purpose we introduce
for $\yy\in\dR^2$ the unitary translation operator
$T_\yy\colon L^2(\dR^2)\arr L^2(\dR^2)$
by $(T_\yy f)(\xx) := f(\xx-\yy)$. Then 
\begin{equation}\label{eq:OP_PI}
	-\Delta_{\aa, \yy+\Lambda} := T^{-1}_\yy(-\Delta_{\aa,\Lambda})T_\yy
\end{equation}	
is the Schr\"odinger operator with point interactions supported on $\yy + \Lambda$. Since $T_\yy$ is a unitary operator,
we have $\sigma(-\Delta_{\aa, \yy+\Lambda}) = \sigma(-\Delta_{\aa, \Lambda})$.

\section{Spectral analysis of the operator $\Theta_r$} \label{section_weighted}

This section is devoted to the proof of 
Theorem~\ref{theorem_spectrum_Theta} 
on the operator $\Th_r$ defined in~\eqref{def_Theta}. 
Since the spectrum of $\Th_r$ is still difficult to investigate, 
we consider instead the spectral problem for the auxiliary family 
of Schr\"odinger operators $\H_{r,\lm}$ in~\eqref{def_H}.
Since $w_r, w_r^{-1} \in L^\infty(\dR^2; \dR)$,
the operator $\H_{r,\lm}$ is well-defined and self-adjoint in $L^2(\dR^2)$ and it holds that 
$\lm \in \s(\Th_r)$ \iff $\lm \in \s(\H_{r,\lm})$ for all $\lm \ge 0$. 

Let the numbers $0 < \lm_1 < \lm_2 < \dots < \lm_N$ be given. 
First, we prove that $\H_{r,\lm_n}$, $n=1,\dots,N$, converges in the 
norm resolvent sense to a Schr\"odinger operator with point interactions supported on $\yy^{(n)} + \Lambda$.
In view of the spectral properties of these Hamiltonians with point interactions 
(summarized in Section~\ref{section_delta}),
it turns out that there exists a lattice $\Lambda$ and constants $c_1, \dots, c_N$ (that appear in the definition of $w_r$)
such that $\lm_n$ belongs to a gap of 
$\s(\H_{r, \lm_n})$. Finally, employing a perturbation argument, we deduce the claim of 
Theorem~\ref{theorem_spectrum_Theta}.

The following theorem treats the convergence of 
$\H_{r, \lm}$ to a Schr\"odinger operator 
with point interactions.
Since the proof of this statement is rather long and technical,
it is postponed to Appendix~\ref{appendix_convergence_delta}.
\begin{thm} \label{thm_convergence_crystal}
	Let $\H_{r,\lm}$, $\lambda \geq 0$, 
	and $-\Delta_{\aa, \yy + \Lambda}$, $\aa \in\dR$,
	$\yy\in\dR^2$, be defined as in~\eqref{def_H} and
	in~\eqref{eq:OP_PI}, respectively,
	and let $\nu \in \dC \sm \dR$.
	Then the following claims hold.
	\begin{myenum}
    \item There exists a constant $\kappa = 
    \kappa(\Lambda, \lm_1, \dots, \lm_N, \Omega, \nu) > 0 $
    such that for any $n \in \{ 1, \dots, N \}$ and all 
    sufficiently small $r > 0$
    \begin{equation*}
    	\big\| (\H_{r,\lm_n} - \nu)^{-1} - 
    	\big( -\Delta_{\aa_n, \yy^{(n)} + \Lambda} - \nu\big)^{-1} \big\| 
       \leq \kappa |\ln r|^{-1},
    \end{equation*}
    where the coefficient $\aa_n$ is given by
	\begin{equation}\label{eq:aa_C}
    	\aa_n = 
    	-c_n \frac{\lm_n |\Omega|}{4 \pi^2} + \frac{C}{2 \pi |\Omega|^2}
    	\qquad\text{with}\qquad
    	C = \int_{\Omega} \int_{\Omega} \ln|x - z| \dd x \dd z.
	  \end{equation}
    \item
    For $\lm \notin \{ \lambda_1, \dots, \lambda_N \}$ 
    there exists a constant $\kappa' = \kappa'(\Lambda, \lm, \Omega, \nu) > 0$
    such that for all sufficiently small $r > 0$
    \begin{equation*}
    	\big\| (\H_{r,\lm} - \nu)^{-1} - (-\Delta - \nu)^{-1} \big\| 
    	\leq \kappa' |\ln r|^{-1}.
    \end{equation*}
  \end{myenum}
\end{thm}


\begin{remark}
  The assumption $\lambda_n \neq \lambda_m$ for $n \neq m$
  is motivated by our application, but it is only technical.
  If we drop this assumption, then one can still
  prove convergence of $\H_{r, \lambda_n}$ to a Schr\"odinger operator with point interactions
  supported on a more complicated lattice with 
  (in general) non-constant interaction strength;
  \cf ~\cite{behrndt_holzmann_lotoreichik_2014}. 
  However, 
  in this case the spectral analysis of the limit operator presents
  a rather difficult problem.
  For special interesting geometries there are results available in the 
  literature~\cite{L16}. 
\end{remark}

Combining the statements of 
Theorem~\ref{thm_convergence_crystal} and of
Proposition~\ref{thm_existance_spectral_gap_delta_op} 
with the perturbation result~\cite[Satz 9.24 b)]{weidmann1},
we obtain the following claim on the spectrum of 
$\H_{r, \lm_n}$.

\begin{prop} \label{proposition_existance_spectral_gap_H}
  Let $0 < \lm_1 < \lm_2 < \dots < \lm_N$,  
  let $a > 0$ be fixed
  and define $\eta := \frac{2 \pi}{|\Omega|} + 1$.
  Let the operator $\H_{r,\lm_n}$ be as in~\eqref{def_H}.
  Then there exist a lattice $\Lambda$ 
  and constants $c_1, \dots, c_N$ 
  (that appear in~\eqref{def_eps})
  such that
  \begin{equation*}
  	\big(\lm_n - \eta - a, \lm_n + \eta + a\big)
  	\subset  \rho(\H_{r,\lm_n}) 
  \end{equation*}
  for all sufficiently small $r > 0$ 
  and all $n \in \{ 1, \dots, N \}$.
\end{prop}
\begin{proof}
	Let $I_n := (\lm_n - \eta - a, \lm_n + \eta + a)$ and
	$J_n := (\lm_n - \eta - 2a, \lm_n + \eta + 2a)$.
	Then, by Proposition~\ref{thm_existance_spectral_gap_delta_op}
	there exists a lattice 
	$\Lambda$ and a coupling constant $\aa \in \dR$ such that 
	\begin{equation*}
    	(\lm_1 - \eta - 2 a, \lm_N + \eta + 2 a) 
    	\subset \rho(-\Delta_{\aa, \Lambda}).
	\end{equation*}
	This implies, in particular, 
	that 	for any $n \in \{ 1, \dots, N \}$ 
	\begin{equation*}
    	\ov{I}_n \subset J_n	\subset \rho(-\Delta_{\alpha, \yy^{(n)} + \Lambda})
        = 
        \rho(-\Delta_{\alpha, \Lambda}),
	\end{equation*}
	where the last equation holds due to translational invariance.
	Next, choose the constants $c_n$ in~\eqref{def_eps} as
	\begin{equation}\label{eq:cn}
    	c_n = \frac{4 \pi^2}{\lm_n |\Omega|} \left( \frac{C}{2 \pi |\Omega|^2} - \alpha \right),
	\end{equation}
	where $C$ is given as in~\eqref{eq:aa_C}. 
	Theorem~\ref{thm_convergence_crystal}\,(i)
	implies that $\H_{r,\lm_n}$ converges in the norm resolvent 
	sense to $-\Delta_{\aa, \yy^{(n)} + \Lambda}$.
	Finally, let  $\E := \E(I_n)$ and $\E_r := \E_r(I_n)$
	be the spectral projections corresponding to the interval 
	$I_n$ and the operators
	$-\Delta_{\aa, \yy^{(n)} + \Lambda}$ 
	and $\H_{r,\lm_n}$, respectively. 
	Since  $\H_{r,\lm_n}$ 
	converges in the norm resolvent sense to 
	$-\Delta_{\aa, \yy^{(n)} + \Lambda}$,
	it follows from \cite[Satz 9.24 b)]{weidmann1} that
  	$\left\| \E - \E_r \right\| < 1$
	for all sufficiently small $r > 0$. 
	Hence, employing \cite[Satz 2.58 a)]{weidmann1} we conclude
	\begin{equation*}
		\dim \ran \E_r = \dim\ran\E =  0
	\end{equation*}
	for all sufficiently small $r > 0$.
	This implies  $I_n \subset \rho(\H_{r,\lm_n})$.
\end{proof}

Now, we are prepared to prove the main result about the spectrum of $\Th_r$. 
\begin{proof}[Proof of Theorem~\ref{theorem_spectrum_Theta}]
	Let $a > 0$ be given. Set 
	$\eta := 2 \pi|\Omega|^{-1} + 1$ and 
	$I_n := (\lm_n - \eta - a, \lm_n + \eta + a)$.
	Choose a lattice $\Lambda$ and the constants $c_1, \dots, c_N$ 
	(that appear in~\eqref{def_eps}) such that
	$I_n \subset \rho(\H_{r,\lm_n})$ for all sufficiently small $r > 0$,
	which is possible
	by Proposition~\ref{proposition_existance_spectral_gap_H}.
	Recall that $\lm \in \rho(\Th_r)$ \iff
	$\lm \in \rho(\H_{r,\lm})$;
	we are going to verify this property for $\lambda$
	belonging to a small neighborhood of~$\lambda_n$.
	Since $I_n \subset \rho(\H_{r,\lm_n})$,	it follows
	from the spectral theorem that
	\begin{equation}\label{eq:Hr_est}
    	\big\| (\H_{r,\lm_n} - \nu) f \big\|_{L^2(\dR^2)} \geq \eta \|f\|_{L^2(\dR^2)}
	\end{equation}
	for all $\nu \in ( \lm_n - a, \lm_n + a)$ and all $f \in H^2(\dR^2)$.
	Note that the definition of $w_r$ in~\eqref{def_eps} implies
	\begin{equation*}
	    \| w_r - 1 \|_{L^\infty} \leq \frac{\eta}{\lm_1} \frac{1}{r^2|\ln r|}
	\end{equation*}
	for $r > 0$ small enough. 
	Therefore, it holds for $\zeta\in\dR$ with 
	$|\zeta| < \lm_1 r^2 |\ln r|$ that
	\begin{equation}\label{eq:est_wr}
	    |\zeta| \|w_r - 1 \|_{L^\infty} 
        < 
        \lm_1 r^2|\ln r|\cdot \frac{\eta}{\lm_1} \frac{1}{r^2|\ln r|}
        = \eta.
	\end{equation}
	For small enough $r > 0$
	we have $|\zeta| < \lm_1 r^2 |\ln r| < a$ 
	and the estimate~\eqref{eq:Hr_est} implies
	for $f \in H^2(\dR^2)$
	\begin{equation*}
    \begin{split}
    	\big\|\H_{r,\lm_n+\zeta}f - (\lm_n+\zeta)f \big\|_{L^2(\dR^2)}
        &= 
        \big\| (-\Delta - (\lm_n + \zeta) (w_r - 1) - (\lm_n + \zeta)) f \big\|_{L^2(\dR^2)} 
        \\[0.4ex]
	    &\geq  
	    \big\| (\H_{r,\lm_n} - (\lm_n + \zeta)) f \big\|_{L^2(\dR^2)}
         -\big\| \zeta (w_r - 1) f \big\|_{L^2(\dR^2)} \\[0.4ex]
		&\geq 
		\left( \eta - |\zeta| \| w_r - 1 \|_{L^\infty} \right) \| f \|_{L^2(\dR^2)}.
	\end{split}
	\end{equation*}
	This and~\eqref{eq:est_wr}
	imply $\lm_n + \zeta \in \rho(\H_{r,\lm_n + \zeta})$, 
	which yields $\lm_n + \zeta \in \rho(\Th_r)$.
\end{proof}
We conclude this section with an explanation how to construct a crystal such that given numbers 
$0 < \lm_1 < \lm_2 \dots < \lm_N$ belong to gap(s) of $\s(\Th_r)$.
First, for a given lattice $\Lambda_0$ 
we choose $\aa_0 \in \dR$ such that
$\aa_0 < g(\bfo,\tt_0) + \frac{1}{2 \pi}(\gamma + \ln 2)$,
where $\tt_0 = -\frac{1}{2}(\sfb_1 + \sfb_2)$ and $\sfb_1$ and $\sfb_2$ are the basis vectors of the dual lattice $\Gamma_0$. 
By Proposition~\ref{thm_spectrum_delta_op}
it holds that $0 \in \rho(-\Delta_{\aa_0, \Lambda_0})$. 
Finding (or estimating) the smallest zeros 
of the function 
\[
	E\mapsto  g(E,\tt) + \frac{1}{2\pi}(\gamma + \ln 2) -  \aa_0 
\]
yields an approximation for the upper and 
the lower endpoints of the bands of the spectrum of
$-\Delta_{\aa_0, \Lambda_0}$; \cf 
Proposition~\ref{thm_spectrum_delta_op}.
Next, choose $k > 0$ as in the proof of 
Proposition~\ref{thm_existance_spectral_gap_delta_op} such that
\[
	(\lm_1 - \eta - 2 a, \lm_N + \eta + 2 a) 
	\subset 
	\rho(-\Delta_{\aa_k, \Lambda_k}),
\]	 
where $\eta = 2 \pi|\Omega|^{-1} + 1$,
$\Lambda_k := k^{-1} \Lambda_0$, $\aa_k = \aa_0 - \frac{\ln k}{2\pi}$,
and $a$ is a small positive constant.
Finally, we define the constants $c_n$
via the formula~\eqref{eq:cn} in the proof of 
Proposition~\ref{proposition_existance_spectral_gap_H} (with $\alpha$ replaced by $\alpha_k$).
Then the crystal that is specified via the lattice $\Lambda_k$ 
and $w_r$ as in~\eqref{def_eps} satisfies
$\{\lm_1, \dots, \lm_N\} \subset \rho(\Th_r)$
for all sufficiently small $r > 0$.

\section{Spectral analysis of the operator $\Ups_r$}
\label{section_div}
In this section we prove that there are no gaps in the spectrum of the operator 
$\Ups_r = -\div (w_r^{-1} \grad f)$ in bounded subsets of $[0, \infty)$, if $r > 0$ is sufficiently small. 
The methods employed in this section are completely different from the methods
in Section~\ref{section_weighted},
partly because the aim is to prove the statement
of an opposite type.
Using the Floquet-Bloch theory for differential operators with periodic coefficients we will see that 
$\s(\Ups_r)$ consists of bands and that the `lowest bands' overlap for small $r > 0$.
The proof of this result is inspired by ideas coming 
from~\cite{figotin_kuchment_1996_1} and makes additionally use 
of a result in~\cite{rauch_taylor_1975} on the convergence of eigenvalues of the Laplace operator on domains with small holes.

First, we set up some notations.
For a fixed $r \ge 0$ we define the sesquilinear form
\begin{equation*}
	\frh_r[f, g] := 
	\big( w_r^{-1} \nabla f, \nabla g \big)_{L^2(\dR^2;\dC^2)}	,
	\qquad \dom \frh_r := H^1(\dR^2),
\end{equation*}
with $w_r$ given by~\eqref{def_eps} for $r > 0$ and $w_r \equiv 1$ for $r = 0$. It is clear that $\frh_r$ is well-defined and symmetric. 
Moreover, by the definition of $w_r$, there exists for any sufficiently small 
fixed $r \ge 0$ a constant $\kappa_r \in (0,1]$ such that
\begin{equation*} 
	0 \le \kappa_r \| \nabla f \|_{L^2(\dR^2;\dC^2)}^2 \le\frh_r[f] \leq \| \nabla f \|_{L^2(\dR^2;\dC^2)}^2
\end{equation*}
for all $f \in H^1(\dR^2)$. This implies that $\frh_r$ is closed. Thus, by the first representation 
theorem~\cite[Thm.~VI~2.1]{kato} there exists a uniquely determined self-adjoint operator associated to the form~$\frh_r$,
which is $\Ups_r$ as in \eqref{def_Upsilon} for $r > 0$ and the free Laplacian $-\Delta$ for $r = 0$.

In order to describe the spectrum of $\Ups_r$, $r \ge 0$,
we use that its coefficients are periodic with respect to the 
lattice $\Lambda$ given by~\eqref{Lambda_and_cell}. 
Let the period cell $\wh\Gamma$ and the Brillouin zone $\wh\Lambda$ 
associated to $\Lambda$ be given by~\eqref{Lambda_and_cell} 
and~\eqref{Brillouin}, respectively, and define for $\tt \in \wh\Lambda$ the 
subspace $\cH(\theta)$ of $L^2(\wh\Gamma)$ as the set of all 
$f \in H^1(\wh\Gamma)$ that satisfy the so-called 
\emph{semi-periodic boundary conditions}, \ie
\begin{equation*}
	\cH(\tt) := \Big\{f\in H^1(\wh\Gamma)\colon
	f( t \sfa_2) = e^{-\ii \tt \sfa_1} f( t \sfa_2 + \sfa_1),
 	f( t \sfa_1) = e^{-\ii \tt \sfa_2} 
 	f( t \sfa_1 + \sfa_2),~t \in [0, 1)\Big\}.
\end{equation*}
Defining now for $r \ge 0$ and $\tt \in \wh\Lambda$ the form
\begin{equation*} 
	\frh_{r, \tt}[f, g] 
	:= \big( w_r^{-1} \nabla f, \nabla g \big)_{L^2(\wh \Gamma;\dC^2)},
	\qquad \dom \frh_{r, \tt} := \cH(\tt),
\end{equation*}
we see, similarly as above, that it satisfies the assumptions of the first representation theorem.
Hence, there exists a uniquely determined self-adjoint operator $\Ups_{r,\tt}$
in~$L^2(\wh\Gamma)$ associated to $\frh_{r, \tt}$.

It is not difficult to see that for all $\tt \in \wh\Lambda$ 
the operator $\Ups_{r,\tt}$, $r \ge 0$,
has a compact resolvent. Hence, its spectrum is purely discrete and we denote
its eigenvalues (counted with multiplicity) by
\begin{equation*}
	0 \leq \lm_{r,1}(\tt) \leq \lm_{r,2}(\tt) \leq 
	\lm_{r,3}(\tt) \leq \dots\\
\end{equation*}
Since $w_r^{-1}$ is periodic, we can 
apply the results from~\cite[Sec. 4]{brown_hoang_plum_wood_2011} 
(\cf also the footnote on p.~3 of \cite{brown_hoang_plum_wood_2011}) and get, 
combined with~\cite{dahlberg_trubowitz_1982} or \cite[Thm.~5.9]{K16}, 
the following characterization for the spectrum of $\Ups_r$.
\begin{prop} \label{proposition_overlap_free}
	For any $r \ge 0$ it holds 
  	\begin{equation*}
    	\s(\Ups_r) = \bigcup_{n=1}^\infty \big[a_{r, n}, b_{r, n}\big],
    	\qquad\qquad
    	a_{r, n} := \min_{\tt \in \wh\Lambda}
    	\lm_{r,n}(\tt),\quad
    	b_{r, n} := \max_{\tt \in \wh\Lambda}
    	\lm_{r,n}(\tt).  	  	
	\end{equation*}
	Moreover, for all $n_0 \in \dN$ there exists $\beta = \beta(n_0) \in (0, 1)$ 
	such that
	$b_{0, n} > (1 + \beta) a_{0, n+1}$	%
	holds for $n = 1,2,\dots,n_0$.
\end{prop}
Our goal is to show that for sufficiently small $r > 0$ 
the relation $b_{r, n} > a_{r, n+1}$ is persisted.
For this purpose, we need the following auxiliary lemma,
which provides a useful estimate for the $L^2$-norm of a function in 
a finite union of disks with radius $r$ in terms of the $H^1$-norm over the whole domain. In what follows it will be convenient to use the notation
\begin{equation}\label{eq:cBs}
	\cB_s := Y +  B_s(\bfo),\qquad s > 0,
\end{equation}	
where $Y$ is as in~\eqref{def:Y}.
\begin{lem} \label{lemma_estimate_forms}
	Let $r > 0$ be sufficiently small and let $\cB_r$ be as in~\eqref{eq:cBs}. 
	Then, there exists a constant $\kappa > 0$ such that
	\begin{equation*}
    	\| f \|_{L^2(\cB_r)}^2 
    	\leq 
    	\kappa r \| f \|_{H^1(\wh\Gamma \sm \cB_r)}^2 
    	+ \kappa r^2 \| f \|_{H^1(\wh\Gamma)}^2
  \end{equation*}
  holds for all $f \in H^1(\wh\Gamma)$.
\end{lem}

\begin{proof}
	Throughout the proof $\kappa > 0$ denotes a generic constant.
	Choose $R > 0$ so small that
	\[
		B_{R}^{(n)} := \ov{B_R(\yy^{(n)})} \subset \wh \Gamma\quad \text{for all}\quad
		n \in \{1, \dots, N\}\qquad\text{and}\qquad
		B_R^{(n)}\cap B_R^{(m)} = \varnothing\quad
		\text{for}	\quad n\ne m.
	\]
	Moreover, assume also that $r \in (0, R)$.
	For fixed $n \in \{ 1, \dots, N\}$ we are going to prove the inequality
	\begin{equation} \label{equation_estimate_one_disk}
	    \| f \|_{L^2(B_{r}^{(n)})}^2 
	    \le  
	    \kappa r \| f \|_{H^1(\wh\Gamma \sm \cB_r)}^2    + \kappa r^2 \| f \|_{H^1(\wh\Gamma)}^2.
	\end{equation}
	By summing up over $n$ the claimed result follows.

	We denote by $\cl(\wh\Gamma)$ the closure of~$\wh\Gamma$.
	Since $C^\infty\big(\cl(\wh\Gamma)\big)$ is dense in $H^1(\wh\Gamma)$, 
	it suffices to prove~\eqref{equation_estimate_one_disk} for smooth functions. 
	Let $f \in C^\infty\big(\cl(\wh\Gamma)\big)$ be fixed.
	We use for its equivalent in polar coordinates $(\rho,\phi)$
	centered at $\yy^{(n)}$ the symbol
	$\sff(\rho, \phi) := f(y^{(n)}_1 + \rho \cos \phi, y^{(n)}_2 + \rho \sin \phi)$,
	where $y^{(n)}_1$ and $y^{(n)}_2$ denote the coordinates of $\yy^{(n)}$.
	Let $\rho > 0$ be such that $\rho < r$.
	Employing the main theorem of calculus we conclude
	\begin{equation*}
    	\sff(\rho, \phi) = 	\sff(r, \phi) - \int_{\rho}^{r} (\p_t \sff)(t, \phi) \dd t.
	\end{equation*}
	This implies 
	\begin{equation} \label{integral_boundary}
    \begin{split}
    	\rho\int_0^{2 \pi} |\sff(\rho, \phi)|^2 \dd \phi
        &
        \le 2 \rho \int_0^{2 \pi} |\sff(r, \phi)|^2 \dd \phi
        + 
        2 \rho \int_0^{2 \pi} \left|\int_{\rho}^{r} 
        (\p_t \sff)(t, \phi) \dd t\right|^2 \dd \phi.
    \end{split}
  	\end{equation}
	Similarly as above, one finds
	\begin{equation} \label{integral_boundary2}
  	\begin{split}
    	\rho \int_0^{2 \pi} |\sff(r, \phi)|^2 \dd \phi
        &
        \le 2 \rho \int_0^{2 \pi} |\sff(R, \phi)|^2 \dd \phi + 2 \rho \int_0^{2 \pi} 
        \left|\int_{r}^{R} (\p_t \sff)(t, \phi) \dd t\right|^2 \dd \phi.
    \end{split}
	\end{equation}
	By the trace theorem~\cite[Thm.~3.37]{mc_lean} applied for the domain $\wh \Gamma \sm \cB_R$ 
	and using that $\wh \Gamma\sm \cB_R \subset \wh \Gamma\sm \cB_{r}$ 
	we get 
	\begin{equation} \label{estimate1}
    \begin{split}
    	\int_0^{2 \pi} |\sff(R, \phi)|^2\rho \dd \phi 
	   	& \le \int_0^{2 \pi} |\sff(R, \phi)|^2 R \dd \phi 
        = \| f|_{\p B_{R}^{(n)}} \|_{L^2( \p B_{R}^{(n)})}^2 \\
        &\le 
        \kappa \| f \|_{H^1(\wh\Gamma \sm \mathcal{B}_{R})}^2 
        \le \kappa \| f \|_{H^1(\wh\Gamma \sm \mathcal{B}_{r})}^2.
   \end{split}
  \end{equation}
  Using the expression for the gradient in polar coordinates and the
  Cauchy-Schwarz inequality, we obtain
  \begin{equation} \label{estimate2}
  \begin{split}
		\rho \int_0^{2 \pi} \left|
		\int_{r}^{R}(\p_t \sff)(t, \phi) \dd t\right|^2 \dd \phi 
        &\leq 
        r R 
        \int_0^{2 \pi} \int_{r}^{R} 
        \left|(\p_t \sff)(t, \phi) \right|^2 \dd t \dd \phi \\
         &
         \leq R \int_0^{2 \pi} \int_{r}^{R} 
         \frac{r}{t} 
         \left|
             (\nabla f)(y^{(n)}_1 + t \cos \phi, y^{(n)}_2 + t \sin \phi) \right|^2 t \dd t \dd \phi \\
      &\leq  \kappa \| f \|_{H^1(\wh\Gamma \sm \mathcal{B}_{r})}^2.
    \end{split}
  \end{equation}
  Equations~\eqref{integral_boundary2},~\eqref{estimate1} and~\eqref{estimate2} imply
	\begin{equation} \label{estimate3} 
	\begin{split}
    	\rho \int_0^{2 \pi} |\sff(r, \phi)|^2 \dd \phi
          &
          \leq \kappa \| f \|_{H^1(\wh\Gamma\sm \mathcal{B}_{r})}^2.
    \end{split}
  \end{equation}
  In a similar way as in~\eqref{estimate2} one shows 
  \begin{equation} \label{estimate4}
    \begin{split}
      \rho \int_0^{2 \pi} \left|\int_{\rho}^{r}
              (\p_t \sff)(t, \phi) \dd t\right|^2 \dd \phi
        & \leq
        r \rho \int_0^{2 \pi} \int_{\rho}^{r} \left|
              (\p_t \sff)(t, \phi) \right|^2 \dd t \dd \phi\\
        & = r \int_0^{2 \pi} \int_\rho^{r} \frac{\rho}{t} \left|
              (\p_t \sff)(t, \phi) \right|^2 t \dd t \dd \phi      
		\le r \| f \|_{H^1(\wh\Gamma)}^2.
    \end{split}
  \end{equation}
	Hence, integrating~\eqref{integral_boundary} from $0$ to $r$ 
	with respect to $\rho$ and using~\eqref{estimate3} 
	and~\eqref{estimate4} we obtain
	\begin{equation*}
	\begin{split}
	  \| f \|_{L^2(B_{r}^{(n)})}^2
	  & =
	  \int_0^{r} \int_0^{2 \pi} |\sff(\rho, \phi)|^2 \rho \dd \phi \dd \rho 
	  \leq
	   \kappa \int_0^{r} \left( \| f \|_{H^1(\wh\Gamma \sm \cB_{r})}^2 +
	    r \| f \|_{H^1(\wh\Gamma)}^2 \right) \dd \rho \\[0.4ex]
	  & = \kappa r \| f \|_{H^1(\wh\Gamma \sm \cB_{r})}^2 + \kappa r^2 
	  \| f \|_{H^1(\wh\Gamma)}^2.\qedhere
	\end{split}
	\end{equation*}
\end{proof}
After these preliminary considerations, we are prepared to show that $\Ups_r$ 
has no gaps in the spectrum in any fixed compact subinterval of $[0,+\infty)$, if $r > 0$ is sufficiently small.
\begin{proof}[Proof of Theorem~\ref{theorem_divergence}]
	Throughout this proof $\kp,\kp',\kp'',\kp''' > 0$ denote generic constants.
	As no confusion can arise, we will use the abbreviation $(\cdot,\cdot)$ 
	for both scalar products $(\cdot,\cdot)_{L^2(\wh\Gamma)}$ and $(\cdot,\cdot)_{L^2(\wh\Gamma;\dC^2)}$, 
	and the shorthand $\|\cdot\|$ for the respective norms $\|\cdot\|_{L^2(\wh\Gamma)}$ 
	and $\|\cdot\|_{L^2(\wh\Gamma;\dC^2)}$.
	
	Fix $L > 0$ and let $n_0 := \min\{n\in\dN_0\colon b_{0,n} > L\}$. 
	Furthermore, we choose $\beta \in (0, 1)$ such that 
	$b_{0, n} > (1 + \beta) a_{0, n+1}$ for all $n \leq n_0$ (note that such a $\beta$ exists 
	by Proposition~\ref{proposition_overlap_free}). 
	Using the min-max principle~\cite[\S XIII.1]{RS-IV} we obtain 
	\begin{equation*}
		\lm_{r, n}(\tt) =
		\min_{\begin{smallmatrix} V \subset \cH(\tt)\\ \dim V = n\end{smallmatrix}} 
		\max_{\begin{smallmatrix} f\in V\\ \|f\| = 1 \end{smallmatrix}} 
		(w_r^{-1} \nabla f, \nabla f)  \leq 
        \min_{\begin{smallmatrix} V \subset \cH(\tt)\\ \dim V = n\end{smallmatrix}} 
		\max_{\begin{smallmatrix} f\in V\\ \|f\| = 1 \end{smallmatrix}}\|\nabla f\|^2 = 
        \lambda_{0, n}(\tt),\qquad\text{for all}\,~\tt \in \wh\Lambda.
	\end{equation*}
	Moreover, let the vectors $\tt_n^\pm \in \wh\Lambda$
	be such that $a_{0, n} = \lm_{0, n}(\tt_n^-)$ and $b_{0, n} = \lm_{0, n}(\tt_n^+)$
	for all $n \in \dN$.
	We aim to prove that $\lm_{r, n+1}(\tt_{n+1}^-) < \lm_{r, n}(\tt_n^+)$ holds for all sufficiently small $r > 0$
	and for any $n < n_0$. In addition, we will show that $\lm_{r, n_0}(\tt_n^+) > L$, which yields then the claim. Fix $n \in\dN$ such that $n < n_0$.
	Choose an $n$-dimensional subspace $W^+ = W^+(n, r)
	\subset \cH(\tt_n^+)$ such that
	\begin{equation*}
    	\lm_{r, n}(\tt_n^+)  = 
	    \min_{\begin{smallmatrix} V \subset \cH(\tt_n^+)\\ \dim V = n\end{smallmatrix}} 
		\max_{\begin{smallmatrix}f\in V\\ \|f\| = 1 \end{smallmatrix}} 
		(w_r^{-1} \nabla f, \nabla f) =
   		\max_{\begin{smallmatrix} f\in W^+\\ \|f\| = 1 \end{smallmatrix}} 
		(w_r^{-1} \nabla f, \nabla f).
	\end{equation*}
	Let $f \in W^+$ with $\| f \| = 1$ and fix $R > 0$ 
	such that $\Omega \subset B_R(\bfo)$.
	Furthermore, we define $\cB := \cB_{rR}$ as in~\eqref{eq:cBs}; in particular, we have $Y + r \Omega \subset\cB$.
	Since $w_r \equiv 1$ on $\wh\Gamma \sm \cB$, we get 
	\begin{equation}\label{eq:estimate}
		( w_r^{-1} \nabla f, \nabla f \big) \ge 
		\|\nabla f\|^2_{L^2(\wh\Gamma \sm \cB;\dC^2)}.
	\end{equation}
	Combining the inequalities $b_{0, n_0} \geq \lm_{0, n}(\tt_n^+) \geq \lm_{r, n}(\tt_n^+)$, 
	the estimate~\eqref{eq:estimate}, and 
	Lemma~\ref{lemma_estimate_forms} we obtain 
	\begin{equation*}
    \begin{split}
    	\| f \|_{L^2(\wh\Gamma \sm \cB)}^2 	& = 
    	\| f \|^2 - \| f \|_{L^2(\cB)}^2\\[0.4ex]
        & \ge1 - \kp r \| f \|_{H^1(\wh\Gamma \sm \cB)}^2 -\kp r^2\|f\|_{H^1(\wh\Gamma)}^2 \\[0.4ex]
		& \ge 
		1 - \kp r(1 + r)  - \kp r \| \nabla f \|_{L^2(\wh\Gamma \sm \cB;\dC^2)}^2 -\kp r^2\|\nabla f\|^2 \\[0.4ex]
        & \ge 
        1 - \kp' r - \kp'' \big(r+|\ln r|^{-1}\big)\cdot
        \big(w_r^{-1}\nabla f,\nabla f\big) \\[0.4ex]
        &\ge 1 - \kp' r - \kp'' \big(r + |\ln r|^{-1} \big) b_{0, n_0} \\
        &\ge 1 -\kp''' |\ln r|^{-1}
    \end{split}
  \end{equation*}
  for all sufficiently small $r > 0$.  Thus, we conclude
  \begin{equation*}	
		(w_r^{-1} \nabla f, \nabla f) \geq 
        \|\nabla f\|_{L^2(\wh\Gamma \sm \cB;\dC^2)}^2  \geq 
        \frac{\|\nabla f\|^2_{L^2(\wh\Gamma \sm \cB;\dC^2)}}{\|f\|_{L^2(\wh\Gamma \sm \cB)}^2} \big(1 - \kp |\ln r|^{-1}\big).
  \end{equation*}
  Taking now the maximum over all normalized
  functions $f \in W^+$ we deduce
  \begin{equation*}
    \begin{split}
  	\lm_{r, n}(\tt_n^+) &=
	 \max_{\begin{smallmatrix}	f\in W^+ \\ \| f \|=1\end{smallmatrix}} 
	 (w_r^{-1} \nabla f, \nabla f)\geq  
	\big(1 - \kp |\ln r|^{-1}\big) 
	\max_{\begin{smallmatrix} f\in W^+ \\ \| f \|=1\end{smallmatrix}} 
    \frac{\|\nabla f\|^2_{L^2(\wh\Gamma\sm\cB;\dC^2)}}{\|f\|_{L^2(\wh\Gamma \sm \cB)}^2}.
    \end{split}
  \end{equation*}
  Note that $\dim\spann\{f|_{\wh\Gamma \sm \cB}\colon f \in W^+ \} = n$ holds for all sufficiently small $r > 0$. 
  Indeed, suppose that this is not the case. Then, there exists  $f \in W^+$, $\|f\| = 1$, 
  such that $\| f \|_{L^2(\wh\Gamma \sm \cB)} = 0$.
  Thus, in view of $\supp f \subset \cB$, Lemma~\ref{lemma_estimate_forms} implies 
  \begin{equation*}
  \begin{split}
  	1 = \| f \|_{L^2(\cB)}^2 \leq \kappa r^2
  	\| f \|_{H^1(\wh\Gamma)}^2 
  	\leq \kappa r^2 + \kappa |\ln r|^{-1} \frh_{r, \tt_n^+}[f] 
        \leq \kappa b_{0, n_0} |\ln r|^{-1},
    \end{split}
  \end{equation*}
  which is a contradiction. Hence, we obtain
  \begin{equation*}
      \frac{\lm_{r, n}(\tt_{n}^+)}{1- \kappa |\ln r|^{-1}} \geq   
        \max_{\begin{smallmatrix}
		f\in W^+\\ \|f\| = 1 \end{smallmatrix}} 
		\frac{\|\nabla f\|_{L^2(\wh\Gamma\sm\cB;\dC^2)}^2}
		{\|f\|^2_{L^2(\wh\Gamma\sm\cB)}} \geq 
     	\min_{\begin{smallmatrix} V \subset \wt\cH(\tt_n^+)\\ 
     		\dim V = n
		\end{smallmatrix}} 
      	\max_{\begin{smallmatrix}
		f\in V\\ f \ne 0 \end{smallmatrix}} 
		\frac{\|\nabla f\|_{L^2(\wh\Gamma\sm\cB;\dC^2)}^2}
		{\|f\|^2_{L^2(\wh\Gamma\sm\cB)}} = 
		\mu_n(\tt_n^+),
  \end{equation*}
  where $\wt\cH(\tt_n^+) := \{ f|_{\wh\Gamma \sm \cB}\colon 
  f \in \cH(\tt_n^+) \}\subset L^2(\wh\Gamma \sm \cB)$ 
  and $\mu_n(\tt_n^+)$ is the $n$-th eigenvalue 
  of the self-adjoint operator in $L^2(\wh\Gamma\sm\cB)$
  associated to the closed, symmetric and densely defined form
  \[
	  \wt\cH(\tt_n^+)\ni f \mapsto \|\nabla f\|^2_{L^2(\wh\Gamma \sm \cB;\dC^2)}.
  \]	  
  The above form corresponds to the Laplace operator in 
  $L^2(\wh\Gamma \sm \cB)$ with semi-periodic boundary conditions on 
  $\p \wh\Gamma$  and Neumann boundary conditions on $\p \cB$.
  Finally, it is known from~\cite[Sec. 3]{rauch_taylor_1975}, 
  that $\mu_n(\tt_n^+)$ converges to $\lm_{0, n}(\tt_n^+) = b_{0,n}$,
  as $r \arr 0+$. Thus, it follows that for sufficiently small $r > 0$
  \begin{equation}\label{eq:overlap1}
		 b_{r,n} = \lm_{r, n}(\tt_n^+) \geq  
         \big(1 - \kappa |\ln r|^{-1}\big) \mu_n(\tt_n^+) \geq 
		 (1 + \beta)^{-1} b_{0,n} >
         a_{0,n+1} \geq 
         \lm_{r, n+1}(\tt_{n+1}^-) = a_{r,n+1}.
  \end{equation}
  Therefore, the first $n_0$ bands in $\s(\Ups_r)$ overlap. It follows by a similar argument that
  \begin{equation}\label{eq:overlap2}
    b_{r, n_0} = \lm_{r, n_0}(\tt_n^+) \ge  \big(1 - \kappa |\ln r|^{-1}\big) \mu_{n_0}(\tt_n^+) > L
  \end{equation}
  for sufficiently small $r > 0$. 
  We deduce from~\eqref{eq:overlap1} and~\eqref{eq:overlap2} the
  claimed inclusion $[0, L] \subset \s(\Ups_r)$.  
\end{proof}

\begin{appendix}

\section{Approximation of Schr\"odinger operators with infinitely many point interactions in $\dR^2$} 
\label{appendix_convergence_delta}

This appendix is devoted to the proof of Theorem~\ref{thm_convergence_crystal}.
Let $N \in \dN$, $\lm_1, \dots, \lm_N \in (0, \infty)$, $Y$, $\Lambda$, $w_r$ 
be as in Subsection~\ref{section_main_result} and let
the self-adjoint operator $\H_{r,\lm}$ be as in~\eqref{def_H}.

Let $\lm > 0$ and let a sufficiently small $r > 0$ be fixed. 
First, we derive a resolvent formula for $\H_{r,\lm}$.
To this aim, we define the set
\begin{equation*}
  \Omega_r := \bigcup_{\yy \in Y + \Lambda} (\yy + r \Omega) \subset\dR^2,
\end{equation*}
and introduce the operators
\begin{equation} \label{def_u}
	u_r \colon L^2(\dR^2) \arr L^2(\Omega_r), \qquad  (u_r f)(\xx) := (w_r (\xx) - 1) f(\xx),
\end{equation}
and
\begin{equation}\label{def_v}
	v_r\colon L^2(\Omega_r) \arr L^2(\dR^2), 
	\qquad
	(v_r f)(\xx) := 
	\begin{cases} 
		f(\xx), & \xx \in \Omega_r, 
		\\ 0,&\text{ else}. 
	\end{cases}
\end{equation}
Note that $\| u_r \| = \| w_r - 1 \|_{L^\infty} = \max_{n \in \{ 1, \dots, N \}} \mu_n\big(|\ln r|^{-1}\big) r^{-2}$ 
and  $\| v_r \| = 1$.  Moreover, the multiplication operator in $L^2(\dR^2)$ associated to 
$(w_r - 1)$ can be factorized as $(w_r-1) = v_r u_r$. Recall that we denote 
$(-\Delta - \nu)^{-1}$, $\nu \in \rho(-\Delta) = \dC \sm [0, \infty)$, by $R_0(\nu)$.
With these notations in hands we can derive an auxiliary resolvent formula for $\H_{r,\lm}$.
\begin{prop} \label{prop:resolvent_formula_1}
	Let $\lm, r > 0$ and $\nu \in \dC \sm \dR \subset \rho(\H_{r,\lm})$ be such that
	$|\Im\nu | > \lm\|w_r - 1\|_{L^\infty}$.
	%
	%
	Then, it holds $1 \in \rho(\lm u_r R_0(\nu) v_r)$ 	and
	\begin{equation}\label{eq:resolvent_identity}
		\big(\H_{r,\lm} - \nu\big)^{-1} 
		= 
		R_0(\nu) + \lm R_0(\nu) v_r\big( 1 - \lm u_r R_0(\nu) v_r \big)^{-1} u_r R_0(\nu).
	\end{equation}
\end{prop}

\begin{proof}
	Note that $\| u_r \| \cdot \| v_r \| = \|w_r - 1\|_{L^\infty}$.
	Thus, by our assumptions on $\nu$ and by the spectral theorem we obtain that $\lm\| u_r R_0(\nu) v_r \| < 1$. 
	Hence, the operator
	\begin{equation*}
		T(\nu) := R_0(\nu) + \lm R_0(\nu) v_r \big( 1 - \lm u_r R_0(\nu) v_r \big)^{-1} u_r R_0(\nu)
	\end{equation*}
	is bounded and everywhere defined in $L^2(\dR^2)$. 
	Moreover, thanks to $\lm (w_r-1) = \lm v_r u_r$ we get for any $f\in L^2(\dR^2)$ that
	\begin{equation*}
	\begin{split}
		\big(\H_{r,\lm} - \nu\big)T(\nu) f 
			& =  \big(-\Delta - \nu - \lm v_r u_r \big)T(\nu) f \\
			& = f + \lm v_r \big( 1 - \lm u_{r} R_0(\nu) v_r \big)^{-1}u_r R_0(\nu) f - \lm v_r u_r R_0(\nu) f \\
	    	& \quad \quad - \lm v_r \big(1 - 1 + \lm u_r R_0(\nu) v_r\big)\big( 1 - \lm u_r R_0(\nu) v_r \big)^{-1}u_r R_0(\nu)f\\
			& = f + \lm v_r \big( 1 - \lm u_r R_0(\nu) v_r \big)^{-1}u_r R_0(\nu) f - \lm v_r u_r R_0(\nu) f \\
			& \quad \quad - \lm v_r \big( 1 - \lm u_r R_0(\nu) v_r \big)^{-1} u_r R_0(\nu) f + \lm v_r u_r R_0(\nu) f = f.
	\end{split}
	\end{equation*}
	Since $\nu \in \dC \sm \dR \subset \rho(\H_{r,\lm})$,  we obtain the resolvent
	identity in~\eqref{eq:resolvent_identity}.
\end{proof}
In order to rewrite the resolvent formula~\eqref{eq:resolvent_identity} in a way which is convenient 
to study its convergence, we set $\scH := \bigoplus_{\yy \in Y + \Lambda} L^2(\Omega)$ and
define the function
\begin{equation*}
	\mu\colon \dR_+ \times (Y + \Lambda) \arr \dR_+,
	\qquad 
	\mu(r,\yy) = \mu_n\big( |\ln r|^{-1} \big) \qquad \text{for}~~\yy\in \yy^{(n)} + \Lambda.
\end{equation*}
Furthermore, for $\nu \in \dC \sm \dR$ we define the operators
$A_r(\nu)\colon \scH \arr L^2(\dR^2)$, $E_r(\nu)\colon L^2(\dR^2)\arr \scH$ by
\begin{subequations}\label{eq:AE}
\begin{align}
	A_r(\nu) \Xi 
	&  := 
	\sum_{\yy \in Y + \Lambda} \int_\Omega G_\nu(\cdot - \yy - r \zz) [\Xi]_\yy(\zz) \dd \zz,\\[0.3ex]
    [E_r(\nu) f]_\yy &:= \int_{\dR^2} G_\nu(r \cdot - \zz + \yy) f(\zz) \dd \zz,
  \end{align}
\end{subequations}
and $B_r(\nu), C_r(\nu), D_r(\nu)\colon \scH \arr \scH$ by 
\begin{subequations}\label{eq:BCD}
	\begin{align}
	[B_r(\nu) \Xi]_\yy 
	& := 
	\mu(r,\yy)\int_\Omega G_\nu(r(\cdot - \zz))[\Xi]_\yy(\zz) \dd \zz, \\[0.3ex]
	[C_r(\nu) \Xi]_\yy 
	& := 
	\sum_{\yy_1 \in (Y + \Lambda) \sm \{\yy\} } \int_\Omega G_\nu(r(\cdot - \zz) + \yy - \yy_1) [\Xi]_{\yy_1}(\zz) \dd \zz, \\[0.3ex]
	[D_r(\nu) \Xi]_\yy & := \mu(r, \yy) [\Xi]_\yy.
\end{align}
\end{subequations}
In the above formulae, $[\Xi]_\yy$ denotes the component of 
$\Xi \in \bigoplus_{\yy \in Y + \Lambda} \cH_\yy$ belonging to $\cH_\yy$,
where $\cH_\yy$ are separable Hilbert spaces. 
To analyse the properties of the operators in~\eqref{eq:AE} and~\eqref{eq:BCD} we require 
several auxiliary unitary mappings: 
the identification mapping $I_r\colon\bigoplus_{\yy \in Y + \Lambda} L^2(\yy + r\Omega) \arr  L^2(\Omega_r)$,
the translation operator  
$T_r\colon \bigoplus_{\yy \in Y + \Lambda} L^2(r \Omega) \arr \bigoplus_{\yy \in Y + \Lambda} L^2(\yy + r \Omega)$,
and the scaling transformation $S_r\colon \scH \arr \bigoplus_{\yy \in Y + \Lambda} L^2(r \Omega)$
defined by
\begin{equation*}
\begin{split}
	(I_r \Xi)(\xx)      &:= [\Xi]_\yy(\xx)\quad (\xx \in \yy + r \Omega, ~\yy \in Y + \Lambda),\\
	[T_r \Xi]_\yy(\xx)  &:= [\Xi]_\yy \left(\xx-\yy\right) \quad (\xx \in \yy + r \Omega), 
	\qquad 
	[S_r \Xi]_\yy (\xx) := \frac{1}{r} \left[ \Xi \right]_\yy \left( \frac{\xx}{r} \right) \quad (\xx \in r \Omega).
\end{split}
\end{equation*}
Note that the inverses of these mappings act as
\begin{equation*} 
\begin{split}
	\big[I_r^{-1} f\big]_\yy(\xx)   &= f(\xx) \quad (\xx \in \yy + r \Omega), \\
	\big[T_r^{-1} \Xi\big]_\yy(\xx) &= [\Xi]_\yy\left(\xx + \yy\right) \quad (\xx \in r \Omega),
	\qquad 
	\big[S_r^{-1} \Xi\big]_\yy(\xx) = r [\Xi]_\yy( r \xx) \quad (\xx \in \Omega).
\end{split}
\end{equation*}
It will also be convenient to define the product
\begin{equation*}
	J_r := r^{-1} I_r T_r S_r.
\end{equation*}
In the following lemma we state some of the basic properties 
of the operators~$A_r(\nu),$ $B_r(\nu), C_r(\nu),$ $D_r(\nu)$ and $E_r(\nu)$.
\begin{lem} \label{lemma_ABCDE_r}
	Let $r > 0$ be sufficiently small and let $\nu \in \dC \sm \dR$.
	Then the following identities are true:
	\begin{equation*}
		A_r(\nu) =  R_0(\nu) v_r J_r,
		\qquad
		B_r(\nu) + D_r(\nu) C_r(\nu)  = 
		J_r^{-1} u_r R_0(\nu) v_r J_r,
		\qquad
		D_r(\nu) E_r(\nu) = J_r^{-1} u_r R_0(\nu).
	\end{equation*}
	In particular, the operators $A_r(\nu), B_r(\nu), C_r(\nu), D_r(\nu)$ 
	and $E_r(\nu)$ are bounded and everywhere defined.
\end{lem}

\begin{proof}
	Let a sufficiently small $r > 0$ and an arbitrary $\nu \in \dC \sm \dR$ be fixed.
	First, we prove the formula $A_r(\nu) =  R_0(\nu) v_r J_r$,
	which automatically implies that $A_r(\nu)$ is bounded and everywhere defined, 
	as the operators $R_0(\nu)$, $v_r$ and $J_r$ separately possess this property. By the definition of $T_r$ and $S_r$
	it follows for $\Xi \in \scH$, $\yy \in Y + \Lambda$, and $\zz \in \yy + r \Omega$ that
	\begin{equation*}
	  	\frac{1}{r} [\Xi]_\yy\left( \frac{\zz - \yy}{r} \right) 
		= 
		[S_r \Xi]_\yy\left( \zz - \yy \right)  = [T_r S_r \Xi]_\yy(\zz).
	\end{equation*}
	Hence, we conclude
	\begin{equation*} 
		\big(R_0(\nu) v_r J_r \Xi\big) (\xx)
		=  
		\sum_{\yy \in Y + \Lambda} \int_{\yy+r\Omega} G_\nu(\xx - \zz)
			\frac{1}{r^2} [\Xi]_\yy\left( \frac{\zz - \yy}{r} \right) \dd \zz. 
	\end{equation*}
	Employing now in each single integral in the above sum a translation 
	$\zeta := \zz - \yy$ and a transformation 
	$\xi := \frac{\zeta}{r}$, we end up with
	\begin{equation} \label{label_operators1}
	\begin{split}
		\big(R_0(\nu) v_r J_r \Xi\big) (\xx)
		&=
		\sum_{\yy \in Y + \Lambda}\int_{\yy + r \Omega} G_\nu(\xx-\zz)\frac{1}{r^2}[\Xi]_\yy\left( \frac{\zz - \yy}{r}\right)\dd \zz \\
		&=
		\sum_{\yy \in Y + \Lambda}\int_{r \Omega} G_\nu(\xx-\zeta-\yy)\frac{1}{r^2} [\Xi]_\yy\left(\frac{\zeta}{r} \right)\dd \zeta \\
		&=
		\sum_{\yy \in Y + \Lambda} \int_\Omega  G_\nu(\xx - r \xi - \yy)  [\Xi]_\yy\left( \xi \right) \dd \xi 
		= 
		(A_r(\nu) \Xi) (\xx).
	\end{split}
	\end{equation}
	Next, we show the identity $D_r(\nu) E_r(\nu) = J_r^{-1} u_r R_0(\nu)$. 
	Indeed, using the definitions of the operators $I_r$, $T_r$, $S_r$, $J_r$ and $u_r$  
	we get for $f\in L^2(\dR^2)$, $\xx \in \Omega$ and $\yy \in Y + \Lambda$
	\begin{equation} \label{label_operators2}
	\begin{split}
		\big[J_r^{-1} u_r R_0(\nu) f \big]_\yy(\xx) 
		&= r^2 \left[T_r^{-1} I_r^{-1} u_{r} R_0(\nu) f \right]_\yy\left(r \xx \right) 
			= r^2 \left[I_r^{-1} u_r R_0(\nu) f \right]_\yy\left(r \xx + \yy\right) \\
		&= \mu(r,\yy)\int_{\dR^2}  G_\nu\big(r \xx + \yy - \zz \big) f(\zz) \dd \zz 
			= [D_r(\nu) E_r(\nu) f]_\yy(\xx).
	\end{split}
	\end{equation}
	Clearly, the operator $D_r(\nu)$ is bounded and everywhere defined. 
	Moreover, since $J_r^{-1}$, $u_r$ and $R_0(\nu)$ 
	are bounded and everywhere defined as well and since 
	$D_r(\nu)$ is boundedly invertible for all sufficiently small $r > 0$, it follows that also $E_r(\nu)$ is bounded.
	
	It remains to prove the identity 
	$B_r(\nu) + D_r(\nu) C_r(\nu)  =  J_r^{-1} u_r R_0(\nu) v_r J_r$.
	As in~\eqref{label_operators2} and~\eqref{label_operators1} we get
	for $\Xi \in \scH$, $\xx \in \Omega$, and $\yy \in Y + \Lambda$ 
	\begin{equation*}
	\begin{split}
		\big[J_r^{-1}  u_r R_0(\nu) v_r J_r\Xi\big]_\yy(\xx) 
		&= 
		\sum_{\yy_1 \in  Y + \Lambda}\mu(r,\yy) \int_\Omega G_\nu(r( \xx - \zz)+ \yy- \yy_1) [\Xi]_{\yy_1}(\zz) \dd \zz \\
		&= \big[\big( B_r(\nu) + D_r(\nu) C_r(\nu) \big) \Xi \big]_\yy(\xx).
	\end{split}
	\end{equation*}
	Finally, since $B_r(\nu)$ and $D_r(\nu)$ are both obviously 
	bounded and everywhere defined 
	due to their diagonal structure and since
	$D_r(\nu)$ is also boundedly invertible for sufficiently small $r$, 
	it follows that $C_r(\nu)$ is also a bounded operator. 
	Thus, the proof of the lemma is complete.
\end{proof}
After all these preparations it is not difficult to transform 
the resolvent formula for $\H_{r,\lambda}$ 
from Proposition~\ref{prop:resolvent_formula_1} 
into another one, which is more convenient for the 
investigation of its convergence. For this purpose, we define 
for $\lm \geq 0$ and $\nu \in \dC \sm \dR$ 
the operator
\begin{equation} \label{def_F}
	F_r(\nu,\lm) := \lm \big(1 - \lm B_r(\nu)\big)^{-1}D_r(\nu).
\end{equation}
Note that $F_r(\nu, \lambda)$ is well defined, as it is known from the one-center case 
that each component of the diagonal operator $1 - \lm B_r(\nu)$
is boundedly invertible; see~\cite[eq.~(5.49) in Chap.~I.5]{albeverio_gesztesy_hoegh-krohn_holden}. 
Hence, thanks to its diagonal structure it is clear that also $(1 - \lm B_r(\nu))^{-1}$
exists as a bounded and everywhere defined operator. 
\begin{thm} \label{theorem_resolvent_formula}
	Let $\lm \geq 0$, $r > 0$, and let $\H_{r,\lm}$ be defined as in~\eqref{def_H}.
	Let $A_r(\nu), E_r(\nu)$ be as in~\eqref{eq:AE}, let
	$B_r(\nu)$, $C_r(\nu), D_r(\nu)$ be as in~\eqref{eq:BCD} 
	and let $F_r(\nu,\lm)$ be given by~\eqref{def_F}.
	Then, for any $\nu \in \dC \sm \dR$ with $\| F_r(\nu,\lm) C_r(\nu) \| < 1$ it holds 
	\begin{equation*}
		(\H_{r,\lm} - \nu)^{-1} 
		= 
		R_0(\nu)  + A_r(\nu) \big[1 -  F_r(\nu,\lm) C_r(\nu) \big]^{-1} F_r(\nu,\lm) E_r(\nu).
	\end{equation*}
\end{thm}
\begin{proof}
	Let $w_r$ be as in~\eqref{def_eps} and the operators $u_r$, $v_r$ be as in~\eqref{def_u},~\eqref{def_v}.
	Choose now a non-real number $\nu$ such that additionally $|\Im  \nu| > \|w_r - 1\|_{L^\infty}$.
	A simple computation shows now
	\begin{equation*}
	\begin{split}
		\big[ 1 -\lm( B_r(\nu) + D_r(\nu) C_r(\nu)) \big]^{-1}  
		&=
		\big[(1 - \lm B_r(\nu))(1 - F_r(\nu,\lm)C_r(\nu))\big]^{-1} \\
		&=
		\big[1 - F_r(\nu,\lm) C_r(\nu)\big]^{-1}(1 - \lm B_r(\nu))^{-1}.
	\end{split}
	\end{equation*}
	Hence, it holds by Proposition~\ref{prop:resolvent_formula_1}
	and Lemma~\ref{lemma_ABCDE_r} 
	\begin{equation*} 
	\begin{split}
		(\H_{r,\lm} - \nu)^{-1} 
		&= 
		R_0(\nu) + \lm R_0(\nu) v_r \big( 1 - \lm u_r R_0(\nu) v_r \big)^{-1} u_r R_0(\nu) \\
		&= 
		R_0(\nu)  + \lm A_r(\nu) J_r^{-1}
			\left[ 1 - \lm J_r\left(B_r(\nu) + D_r(\nu) E_r(\nu)\right)J_r^{-1}\right]^{-1} J_r D_r(\nu) E_r(\nu) \\
		& = 
		R_0(\nu) +\lm A_r(\nu) 
			\left[ 1 - \lm (B_r(\nu) + D_r(\nu) E_r(\nu)) \right]^{-1}D_r(\nu) E_r(\nu)\\
		&= R_0(\nu)  + A_r(\nu) \big[1 -  F_r(\nu,\lm) C_r(\nu) \big]^{-1} F_r(\nu,\lm) E_r(\nu).
	\end{split}
	\end{equation*}
	For general $\nu \in \dC \sm \dR$ with $\| F_r(\nu,\lm) C_r(\nu) \| < 1$
	the statement follows by analytic continuation.
\end{proof}
Now we have all the tools to analyse the convergence of $\H_{r,\lm}$
in the norm resolvent sense.
For this purpose, it is sufficient to compute the limits of the operators 
$A_r(\nu), C_r(\nu), E_r(\nu)$ and $F_r(\nu,\lm)$ separately.
The obvious candidates for the limits of $A_r(\nu), C_r(\nu)$ and $E_r(\nu)$, as $r \arr 0+$,
are given by $A_0(\nu), C_0(\nu)$, and $E_0(\nu)$ that are  defined 
as in~\eqref{eq:AE} and~\eqref{eq:BCD} with $r = 0$. 
The convergence of $F_r(\nu,\lm)$ is more subtle, as
$G_\nu(0)$ is not defined. 
The known analysis of the convergence in the one-center case
\cite[Chap.~I.5]{albeverio_gesztesy_hoegh-krohn_holden} suggests the following limit operator:
\begin{equation} \label{def_F0}
    F(\nu,\lm)\colon \scH \arr \scH, 
    \qquad
    \big[F(\nu,\lm) \Xi]_\yy(\xx) 
    := 
    \frac{q(\yy,\nu,\lm)}{|\Omega|^{2}} \langle [\Xi]_\yy \rangle_\Omega,
\end{equation}
where $\langle f \rangle_\Omega = \int_\Omega f \dd \xx$ and $q(\yy,\nu,\lm)$ is given by
\begin{equation*}
	q(\yy,\nu,\lm) = 
	\begin{cases}
		2\pi \Big\{\ln \frac{\sqrt{\nu}}{2\ii}  - \gamma + 2\pi \aa_n \Big\}^{-1}, & 
		\lm = \lm_n, \yy\in \yy^{(n)} + Y,~n\in \{1,\dots,N\},\\
		0, & \text{else,} 
	\end{cases}
\end{equation*}
with $\aa_n$ as in~\eqref{eq:aa_C}. 
Before going further with the proof of the convergence of $(\H_{r,\lm} - \nu)^{-1}$,
we recall the asymptotics of the integral kernel $G_\nu(\xx-\yy)$ of $R_0(\nu)$.
In a way similar to~\cite[Prop. A.1]{BEHL16}, one can prove the following claim.
\begin{lem} \label{proposition_asymptotics_G_lambda}
	Let $\nu \in \dC \sm \dR$
	and let $G_\nu(x) =  \frac{\ii}{4} H^{(1)}_0\big(\sqrt{\nu} |x|\big)$
	be as in \eqref{def_G_lambda}.
	Then, there exist constants $\rho = \rho(\nu) > 0$,
	$\kp = \kp(\nu) > 0$, $K = K(\nu) > 0$, $\kp' = \kp'(\nu)  > 0$
	and $K' = K'(\nu) > 0$ such that
	\begin{equation*}
	    \big| G_\nu(\xx) \big| 
	    \leq  
	    \begin{cases}
		    \kp \big( 1 + \big| \ln |\xx| \big| \big), &|\xx| \leq \rho,\\
		    K e^{-\Im\sqrt{\nu} |\xx|},				   &|\xx| \geq \rho,
	    \end{cases}
	    \qquad
	    \big| \nabla G_\nu(\xx) \big| 
	    \leq 
	    \begin{cases}
		    \kp' |\xx|^{-1},               &  |\xx| \leq \rho,\\
		     K'  e^{-\Im\sqrt{\nu} |\xx|}, &  |\xx| \geq \rho.
	    \end{cases} 
	\end{equation*}
	In particular, $G_\nu$ and $\nabla G_\nu$ are integrable functions.
\end{lem}
Now, we are prepared to investigate the convergence of 
$A_r(\nu), C_r(\nu), E_r(\nu)$, and $F_r(\nu,\lm)$, as $r\arr 0+$.
\begin{lem} \label{lemma_convergence}
	Let $\nu \in \dC \sm \dR$, 
	let the operators $A_r(\nu)$, $C_r(\nu)$, $E_r(\nu)$ be defined as in~\eqref{eq:AE} 
	and~\eqref{eq:BCD}. Let the operators $F_r(\nu,\lm)$ and $F(\nu,\lm)$ be as in~\eqref{def_F}
	and~\eqref{def_F0}, respectively.
	Then there exists a constant $M = M(\nu, \lm) > 0$  such that
	\begin{equation*}
	\begin{split}
		&\big\| A_r(\nu) - A_0(\nu) \big\| \leq M r^{1/4}, 
		\qquad 
	   \big\| C_r(\nu) - C_0(\nu) \big\| \leq M r, \\
		&\big\| E_r(\nu) - E_0(\nu) \big\| \leq M r^{1/4}, 
		\qquad     
		\big\| F_r(\nu,\lm) - F(\nu, \lm) \big\| \leq M |\ln r|^{-1},
	\end{split}
	\end{equation*}
	for all sufficiently small $r > 0$.
\end{lem}

\begin{proof}
	Let $\nu \in \dC \sm \dR$. First, we analyze convergence of $E_r(\nu)$. 
	For $f \in L^2(\dR^2)$ we get, using the Cauchy-Schwarz inequality,
	\begin{equation*}
	\begin{split}
		\big\| (E_r(&\nu) - E_0(\nu)) f \big\|^2_{\scH} 
		=
		\sum_{\yy \in Y + \Lambda}\int_\Omega 
		\left| \int_{\dR^2} \left( G_\nu(r \xx - \zz + \yy) - G_\nu(\zz - \yy) \right)f(\zz)\dd \zz\right|^2 \dd \xx 	
		\\
		&\le
		\sum_{\yy \in Y +  \Lambda}	\int_\Omega 
		\left(\int_{\dR^2}\left|G_\nu(r \xx - \zz + \yy) - G_\nu(\zz - \yy)\right|^2e^{\Im \sqrt{\nu} |\zz - \yy|} \dd \zz 
		\cdot 
		\int_{\dR^2} e^{-\Im \sqrt{\nu} |\zz - \yy|} |f(\zz)|^2\dd \zz\right) \dd \xx \\
		&= 
		\bigg(\int_\Omega \int_{\dR^2}\left|G_\nu(\zz - r \xx) - G_\nu(\zz) \right|^2 e^{\Im \sqrt{\nu} |\zz|} \dd \zz \dd \xx \bigg)
		\cdot
		\bigg(\int_{\dR^2} \sum_{\yy \in Y + \Lambda} e^{-\Im\sqrt{\nu} |\zz - \yy|} |f(\zz)|^2 \dd \zz\bigg).
	\end{split} 
	\end{equation*}
	The term $\sum_{\yy \in Y + \Lambda} e^{-\Im \sqrt{\nu} |\zz - \yy|}$ is uniformly bounded in $\zz$, 
	as this sum can be estimated by a convergent $\zz$-independent geometric series. 
	In fact, one can find for each $\zz \in \dR^2$ 
	points $\wh \zz\in\wh \Gamma$
	and $\wh \yy \in Y$ with $\zz = \wh \zz + \wh \yy$. Then, it holds because of the 
	periodicity of $\Lambda$
	\begin{equation*}
	  \begin{split}
	    \sum_{\yy \in Y + \Lambda} e^{-\Im \sqrt{\nu} |\zz - \yy|} 
	      &= \sum_{\yy \in Y + \Lambda} e^{-\Im \sqrt{\nu} |\wh \zz + \wh \yy - \yy|}
	      \leq e^{\Im \sqrt{\nu} |\wh \zz|} \sum_{\yy \in Y + \Lambda} e^{-\Im \sqrt{\nu} |\wh \yy - \yy|} \\
	      &\leq \kappa \sum_{\yy \in Y + \Lambda} e^{-\Im \sqrt{\nu} |\yy|}
	      < \infty,
	  \end{split}
	\end{equation*}
	where the last sum is independent of $\zz$. Next,
	using the mean value theorem we obtain that for almost all $(\xx, \zz) \in \Omega \times \dR^2$ 
	\begin{equation*}
		G_\nu(\zz - r \xx) - G_\nu(\zz) = -\int_0^1 \nabla G_\nu(\zz - r \tt \xx) \cdot r \xx \dd\tt
	\end{equation*}
	holds. This implies 
	\begin{equation} \label{estimate_main_theorem1}
	\begin{split}
		\int_\Omega \int_{\dR^2} \left| G_\nu(\zz - r \xx) - G_\nu(\zz) \right| \dd \zz \dd \xx
		&\leq 
		M_1 \int_\Omega \int_{\dR^2} \int_0^1 r 
		\left| \nabla G_\nu(\zz - r \tt \xx) \right| \dd \tt \dd \zz \dd \xx \\
		&= 
		r M_1 |\Omega| \int_{\dR^2} \left| \nabla G_\nu(\zz) \right| \dd \zz = M_2 r,
	\end{split}
	\end{equation}
	where we used the translational invariance of the Lebesgue measure and
	that $\nabla G_\nu$ is integrable by Lemma~\ref{proposition_asymptotics_G_lambda}. 
	Hence, it follows with the help of the Cauchy-Schwarz inequality
	\begin{equation*}
	\begin{split}
		\big\| E_r(\nu)& - E_0(\nu) \big\|^4 
		\leq 
		M_3
		\left(\int_{\Omega} \int_{\dR^2} 
		\left| G_\nu(\zz - r \xx) - G_\nu(\zz) \right|^2 e^{\Im \sqrt{\nu} |\zz|}  \dd \zz \dd \xx\right)^2  \\
		&\leq 
		M_3
		\left(\int_\Omega \int_{\dR^2}\left| G_\nu(\zz - r \xx) - G_\nu(\zz) \right|^3 e^{2 \Im \sqrt{\nu} |\zz|}\dd\zz \dd \xx\right) 
		\cdot
		\left(\int_\Omega \int_{\dR^2}\left| G_\nu(\zz - r \xx) - G_\nu(\zz) \right|  \dd \zz \dd \xx\right) \\
		& 
		\leq M_2 M_3 r\left(\int_\Omega \int_{\dR^2}\left| G_\nu(\zz - r \xx) - G_\nu(\zz) \right|^3 
			e^{2 \Im \sqrt{\nu} |\zz|}\dd \zz\dd \xx\right).
	\end{split}
	\end{equation*}
	Employing the triangle inequality and the estimates of $G_\nu$ in Lemma~\ref{proposition_asymptotics_G_lambda},
	we see that the last integral is finite and we end up with $\| E_r(\nu) - E_0(\nu) \big\| \le Mr^{1/4}$.
	A similar argument yields $\| A_r(\nu)^* - A_0(\nu)^* \| \leq M r^{1/4}$
	and therefore, $\| A_r(\nu) - A_0(\nu) \| \leq M r^{1/4}$.
	
	Next, we analyze the convergence of $C_r(\nu)$. Using that the Hilbert-Schmidt norm of an integral 
	operator is an upper bound for its operator norm and a symmetry argument, one sees similarly as in the appendix of~\cite{holden_hoegh_krohn_johannesen1984} 
	\begin{equation*}
		\big\| C_r(\nu) - C_0(\nu)\big\| 
		\leq 
		\sup_{\yy \in Y + \Lambda} \sum_{\yy_1 \in (Y + \Lambda)\sm \{ \yy \}} 
		\left(\int_\Omega \int_\Omega  \left|G_\nu(r (\xx - \zz) + \yy - \yy_1) - 
											G_\nu(\yy - \yy_1) \right|^2 \dd \zz \dd \xx\right)^{1/2}. 
	\end{equation*}
	%
	In the same way as in~\eqref{estimate_main_theorem1}, we get
	\begin{equation*}
		|G_\nu(r (\xx - \zz) + \yy - \yy_1) - G_\nu(\yy - \yy_1)| 
		    \leq M_4 r \int_0^1 |\nabla G_\nu(r \tt (\xx - \zz) + \yy - \yy_1) | \dd \tt.
	\end{equation*}
	By the estimates in Lemma~\ref{proposition_asymptotics_G_lambda}~(ii),
	we get using that $\yy \neq \yy_1$
	\begin{equation*}
		|\nabla G_\nu(r \tt (\xx - \zz) + \yy - \yy_1) | \leq \kp e^{-\Im \sqrt{\nu} |\yy - \yy_1|}
	\end{equation*}
	for all sufficiently small $r > 0$ and all $\xx, \zz\in\Omega$. Hence, we deduce
	\begin{equation*}
		\left( \int_\Omega \int_\Omega
		\left|G_\nu(r (\xx - \zz) + \yy - \yy_1) - G_\nu(\yy - \yy_1) \right|^2 
		\dd \zz \dd \xx \right)^{1/2}
		\leq 
		M_5 r  e^{-\Im\sqrt{\nu} |\yy - \yy_1|}.
	\end{equation*}
	%
	This implies
	\begin{equation*}
		\big\| C_r(\nu) - C_0(\nu) \big\|
		\leq 
		\sup_{\yy_1 \in Y + \Lambda} \sum_{\yy \in (Y + \Lambda)\sm \{ \yy_1 \}} 
			M_5 r  e^{-\Im\, \sqrt{\nu} |\yy - \yy_1|} \leq M r,
	\end{equation*}
	where we  estimated the last sum by a convergent geometric series. 
	Thus, the claim on the convergence of $C_r(\nu)$ is shown.
	
	It remains to analyze the convergence of $F_r(\nu, \lambda)$. 
	For this purpose, we define for $n \in \{1, \dots, N \}$ the bounded auxiliary operator 
	$\wt B_{r, n}(\nu)$ in $L^2(\Omega)$ via
	\begin{equation*} 
		\big(\wt B_{r, n}(\nu) f\big)(\xx) 
		= 
		\mu_n\big( |\ln r|^{-1} \big)\int_\Omega G_\nu(r(\xx - \zz)) f(\zz) \dd \zz.
	\end{equation*}
	From the one-center case~\cite[Chap. I.5]{albeverio_gesztesy_hoegh-krohn_holden}
	we know\footnote{Note that an inverse sign is missing in eq.~(5.61)
	in~\cite[Chap.~I.5]{albeverio_gesztesy_hoegh-krohn_holden}.}
	\begin{equation*}
	\begin{split}
		\big(\big(1 - \lambda \wt B_{r, n}(\nu)\big)^{-1} f\big)(\xx) 
		=
		\frac{1}{2 \pi |\Omega|}\cdot|\ln r|\cdot q(\yy_n,\nu,\lm_n) \cdot
			\langle f \rangle_\Omega + \cO(1),\qquad r\arr 0+. 
	\end{split}
	\end{equation*}
	Using this and the diagonal structure of $(1 - \lm B_r(\nu))^{-1} D_r(\nu)$ 
	it follows immediately that
	\begin{equation*}
		\big\| F_r(\nu, \lm) - F(\nu, \lm) \big\| 
		= 
		\big\| \lm (1 - \lm B_r(\nu))^{-1} D_r(\nu) - F(\nu, \lm) \big\| \leq M |\ln r|^{-1}
	\end{equation*}
	for all sufficiently small $r > 0$.  This finishes the proof of the lemma.
\end{proof}
Since we know now the convergence properties of all the involved operators 
in the resolvent formula of $\H_{r,\lm}$, 
we are ready to prove Theorem~\ref{thm_convergence_crystal}.
\begin{proof}[Proof of Theorem~\ref{thm_convergence_crystal}]
	We split the proof of this theorem into three steps. 
	First, we show that $\big( \H_{r,\lm} - \nu \big)^{-1}$ converges in the operator norm, 
	if the imaginary part of $\nu \in \dC \sm \dR$ has a sufficiently large
	absolute value.
	Then, in the second step we prove that the limit operator is indeed the Sch\"odinger
	operator with point interactions specified as in the theorem. Finally, we extend this 
	convergence result to any $\nu \in \dC \sm \dR$.

	\noindent {\it Step 1.}
	Let the operators $A_r(\nu), C_r(\nu)$, and $E_r(\nu)$  
	be defined as in~\eqref{eq:AE} and~\eqref{eq:BCD}, 
	let $F_r(\nu,\lm)$ be as in~\eqref{def_F} and let $F(\nu,\lm)$ be given by~\eqref{def_F0}.
	Fix $\nu\in\dC$ with $|\Im\nu|$ so large that 
	$Q_{\aa, \yy^{(n)} + \Lambda}(\nu)$ given by~\eqref{def_Gamma} has a bounded and 
	everywhere defined inverse and that $\| F(\nu,\lm) C_0(\nu) \| < 1$. 
	Note that such a choice is possible, as 
	\begin{equation*}
		\| F(\nu,\lm) \| \leq \frac{|q(\yy,\nu,\lm)|}{|\Omega|^2} \arr 0,
		\qquad 
		|\Im \nu| \arr \infty,
	\end{equation*}
	and 
  \begin{equation*}
      \big\| C_0(\nu)\big\| 
     \leq 
      \sup_{\yy \in Y + \Lambda} \sum_{\yy_1 \in (Y + \Lambda)\setminus \{ \yy \}} \left( 
      \int_{\Omega} \int_{\Omega} 
           \left|G_{\nu}(\yy - \yy_1) \right|^2 \dd \zz \dd \xx\right)^{1/2}, 
  \end{equation*}
  where we used the Holmgren bound for the operator norm of $C_0(\nu)$ 
  from the appendix of~\cite{holden_hoegh_krohn_johannesen1984}, a symmetry argument 
  and that the Hilbert-Schmidt norm of an integral 
  operator is an upper bound for its operator norm.
  Because of the asymptotics of $G_\nu$ from Lemma~\ref{proposition_asymptotics_G_lambda}
  the last sum can be estimated 
  by a convergent geometric series uniformly in $|\Im \nu|$,
  which yields finally the justification of our assumption.
	
	Employing \cite[Thm.~IV~1.16]{kato} and Lemma~\ref{lemma_convergence} we deduce 
	that $[1 - F_r(\nu,\lm) C_r(\nu)]^{-1}$ is a bounded and everywhere defined operator and
	\begin{equation} \label{convergence_inverse}
	\begin{split}
		\big\| [1 - F_r(\nu,\lm)& C_r(\nu)]^{-1}  - [1 - F(\nu,\lm)C_0(\nu)]^{-1} \big\| \\
		&\leq 
		\frac{\| [1 - F(\nu,\lm)C_0(\nu)]^{-1} \|^2
			\cdot \|F_r(\nu,\lm)C_r(\nu) - F(\nu,\lm)C_0(\nu)\|}{1 - \|F_r(\nu,\lm)C_r(\nu) - F(\nu,\lm)C_0(\nu)\| 
					\cdot \|[1 - F(\nu,\lm)C_0(\nu)]^{-1} \|}  
	    \leq M |\ln r|^{-1}
	\end{split}
	\end{equation}
	for some constant $M > 0$. This together with Theorem~\ref{theorem_resolvent_formula} 
	and Lemma~\ref{lemma_convergence} yields eventually
	\begin{equation} \label{equation_limit}
	\begin{split}
		\lim_{r \arr 0+} \big( \H_{r,\lm} - \nu \big)^{-1} 
		&= 
		\lim_{r \arr 0+} 
	    \big[ R_0(\nu)  + A_r(\nu) [1 - F_r(\nu,\lm)C_r(\nu)]^{-1} F_r(\nu,\lm)  E_r(\nu) \big] \\
		&= 
		R_0(\nu) + A_0(\nu) [1 - F(\nu,\lm) C_0(\nu)]^{-1} F(\nu,\lm) E_0(\nu).
	\end{split}
	\end{equation}
	Moreover, Lemma~\ref{lemma_convergence} and \eqref{convergence_inverse} 
	imply that the order of convergence is $|\ln r|^{-1}$.
	Note that $F(\nu, \lm) = 0$, if $\lm \notin \{ \lm_1, \dots, \lm_N \}$. 
	Therefore, the above considerations are true for any $\nu \in \dC \sm \dR$
	and item~(ii) of this theorem follows.
	
	\noindent {\it Step 2.} 
	From now on assume that $\lm = \lm_n$ for some $n \in \{ 1, \dots, N \}$.
	We are going to prove that 
	the limit operator in~\eqref{equation_limit} is equal to 
	$(-\Delta_{\aa_n, \yy^{(n)} + \Lambda} - \nu)^{-1}$ 
	with the coupling constant $\aa_n$ as in the formulation of the theorem.  
	For that purpose we set 
	$\cH := \bigoplus_{\yy \in Y + \Lambda} \dC = \ell^2(Y + \Lambda)$ 
	and introduce the bounded and everywhere defined 
	operators $U \colon \cH \arr \scH$ and $V\colon \scH \arr \cH$ via
	\begin{equation*}
		[ U \xi ]_\yy :=  \frac{q(\yy, \nu, \lm)}{|\Omega|} [\xi]_\yy 
		\quad \text{and} \quad
		[ V \Xi ]_\yy :=  \left\langle [\Xi]_\yy \right\rangle_\Omega
	\end{equation*}
	and the operators $G\colon L^2(\dR^2)\arr \cH$ and $H\colon \cH\arr \cH$ via
	\begin{equation*}
		[Gf]_\yy = \int_{\dR^2} G_\nu(\zz - \yy) f(\zz) \dd \zz
		\quad\text{and}\quad 
		[ H\xi ]_\yy 
		:= 
		\sum_{\yy_1 \in Y + \Lambda} \wt{G}_\nu(\yy - \yy_1) [\xi]_{\yy_1},
	\end{equation*}
	where $\wt{G}_\nu$ is given by~\eqref{def_G_tilde}.
	Note that the operator $H$ is bounded. Indeed, the Holmgren bound 
	(\cf the appendix of \cite{holden_hoegh_krohn_johannesen1984}) and a symmetry argument imply
	\begin{equation*}
		\| H \| 
		\leq 
		\sup_{\yy \in Y + \Lambda} \sum_{\yy_1 \in Y + \Lambda} 
		\big|\wt{G}_\nu(\yy - \yy_1)\big|.
	\end{equation*}
	Thanks to the estimates of $G_\nu$ in Lemma~\ref{proposition_asymptotics_G_lambda}
	the last sum is bounded by a convergent geometric series
	which can be estimated by a value independent of $\yy \in Y + \Lambda$.
	
	We find for $f \in L^2(\dR^2)$ and $\xx \in \Omega$
	\begin{equation*}
		[F(\nu, \lm) E_0(\nu) f]_\yy(\xx) 
		= \frac{q(\yy, \nu, \lm)}{|\Omega|} \int_{\dR^2} G_\nu(\zz - \yy) f(\zz) \dd \zz
		= \big[ U Gf \big]_\yy.
	\end{equation*}
	Similarly, it holds for any $\Xi \in \scH$ 
	\begin{equation*}
		[F(\nu, \lambda) C_0(\nu) \Xi]_\yy 
		= 
		\frac{q(\yy,\nu,\lm)}{|\Omega|}
		\sum_{\yy_1 \in Y + \Lambda}\wt{G}_\nu(\yy_1 - \yy) 
        \langle [\Xi]_{y_1}\rangle_\Omega
		= 
		\big[ U H V \Xi \big]_\yy.
	\end{equation*}
	Finally, we see for $\Xi \in \scH$
	\begin{equation*}
		\big(A_0(\nu) \Xi\big)(\xx) 
		= \sum_{\yy \in Y + \Lambda} G_\nu(\xx - \yy) \langle [\Xi]_\yy \rangle_\Omega
		= \sum_{\yy \in Y + \Lambda} G_\nu(\xx - \yy) \big[ V \Xi \big]_\yy.
	\end{equation*}
	This implies 
	\begin{equation*} 
	\begin{split}
		R_0(\nu) + A_0(\nu) [1 - F(\nu,\lm) C_0(\nu)]^{-1} F(\nu,\lm) E_0(\nu)
		&
		=
		R_0(\nu) + \sum_{\yy \in Y + \Lambda} G_\nu(\xx - \yy) 
		\big[ V \big( 1 - U H V \big)^{-1} U Gf \big]_\yy.\\
	\end{split}
	\end{equation*}
	In order to simplify the last formula, we have to investigate the operator 
	$1 - VUH$.
	Set $\cH_1 := \ell^2(\yy^{(n)} + \Lambda)$ and 
	$\cH_2 := \ell^2((Y \sm \yy^{(n)}) + \Lambda)$. The decompositions of the operators 
	$VU$ and $1 - VUH$ with respect to $\cH = \cH_1 \oplus \cH_2$ are given by
	\begin{equation*}
		VU = \begin{pmatrix} q(\yy, \nu, \lm) I_{\cH_1} & 0 \\ 0 & 0 \end{pmatrix}
		\quad \text{and} \quad 
		1 - VUH = \begin{pmatrix} q(\yy, \nu, \lm) Q & B \\ 0 & I_{\cH_2} \end{pmatrix},
	\end{equation*}
	where $Q : = Q_{\aa_n, \yy^{(n)} + \Lambda}(\nu)$ is the operator which is defined 
	via the matrix~\eqref{def_Gamma_alpha} and $B\colon \cH_2 \arr \cH_1$ is a
	bounded operator which does not need to be specified because it will cancel
	in the further computations.
	Recall that due to our assumptions on $\nu$ the operator $Q$ and 
	hence also  $1 - VUH$ are both boundedly invertible. A simple calculation shows 
	\begin{equation*}
		\big(1 - VUH \big)^{-1}
		= 
		\begin{pmatrix} 
			q(\yy, \nu, \lm)^{-1} Q^{-1} & - q(\yy, \nu, \lm)^{-1} Q^{-1} B \\ 
				   0                          & I_{\cH_2} 
	  \end{pmatrix}.
	\end{equation*}
	Therefore, we find 
	\begin{equation*}
		\big( 1 - VUH  \big)^{-1} VU
		= 
		\begin{pmatrix} 
		  Q^{-1} & 0 \\ 
		 0      & 0 
		\end{pmatrix}.
	\end{equation*}
	Recall that the points in $\Lambda$ are denoted by 
	$\yy_p = p_1 \sfa_1 + p_2 \sfa_2$, $p = (p_1, p_2)^\top \in \dZ^2$,
	where $\sfa_1$ and $\sfa_2$ span the basis cell $\wh\Gamma$, 
	and the elements of the infinite matrix $Q^{-1}$
	are denoted by $r^{ml}_{\aa_n,\yy^{(n)} + \Lambda}(\nu)$.
	Then, we deduce for $f \in L^2(\dR^2)$
	\begin{equation*}
	\begin{split}
	  \lim_{r \arr 0+} 
	  \big( \H_{r,\lm} - \nu \big)^{-1} f
	  &= R_0(\nu) f + A_0(\nu) [1 - F(\nu,\lm) C_0(\nu)]^{-1} F(\nu,\lm) E_0(\nu) f \\
	  &= R_0(\nu) + \sum_{\yy \in Y + \Lambda} G_\nu(\xx - \yy) 
		\big[ V \big( 1 - U H V \big)^{-1} U Gf \big]_\yy\\
	  &= R_0(\nu) + \sum_{\yy \in Y + \Lambda} G_\nu(\xx - \yy) 
		\big[ \big( 1 - V U H \big)^{-1} V U Gf \big]_\yy\\
	  &= R_0(\nu) f + \sum_{m, l \in \dZ^2} 
	      r^{ml}_{\aa_n,\yy^{(n)} + \Lambda}(\nu) 
	       \big(\,f,\, \ov{G_{\nu}(\cdot - 	\yy^{(n)} - \yy_l)}\big)_{L^2(\dR^2)} 
	         	G_{\nu}(\cdot - \yy^{(n)}- \yy_m) \\
	  & = (-\Delta_{\aa_n, \yy^{(n)} + \Lambda} - \nu)^{-1} f,
	\end{split}
	\end{equation*}
	which is the desired result for our special choice of $\nu$.
	
	\noindent {\it Step 3.}
	Finally, we extend the result from {\it Step 2} to any $\wt\nu \in \dC \sm \dR$. 
	With the shorthand 
	$D(\nu) := ( \H_{r,\lm} - \nu)^{-1} - (-\Delta_{\aa_n, \yy^{(n)} + \Lambda} - \nu)^{-1}$ we get via a simple calculation
	\begin{equation*}
	  D(\wt\nu)
	  = \big[ 1 + (\wt{\nu} - \nu) \big(-\Delta_{\aa_n, \yy^{(n)} + \Lambda} - \wt{\nu}\big)^{-1} \big]\cdot D(\nu)\cdot
	  \big[ 1 + (\wt{\nu} - \nu) \big( \H_{r,\lm} - \wt{\nu} \big)^{-1} \big].
	\end{equation*}
	Thus, the claimed convergence result is true for any $\wt\nu \in \dC \sm \dR$ 
	and the order of convergence is $|\ln r|^{-1}$.
	This finishes the proof of Theorem~\ref{thm_convergence_crystal}.
\end{proof}

\end{appendix}

\section*{Acknowledgment}

The authors thank S. Albeverio, J. Behrndt, P. Exner, F. Gesztesy, and D. Krej\v{c}i\v{r}\'{i}k for useful 
hints to solve the approximation problems. 
Moreover, M. Holzmann acknowledges 
financial support under a scholarship of the program ``Aktion Austria - Czech Republic''
during a research stay in Prague by the Czech Centre for International Cooperation in Education (DZS) 
and the Austrian Agency for International Cooperation in Education and Research (OeAD). V.~Lotoreichik was supported by the grant No.\ 17-01706S
of the Czech Science Foundation (GA\v{C}R).

\bibliographystyle{alpha}

\newcommand{\etalchar}[1]{$^{#1}$}


\end{document}